\documentclass[11pt]{amsart}

\usepackage{amsfonts,amssymb,stmaryrd,amscd,amsmath,latexsym,amsbsy,hyperref}
\usepackage{xcolor}
\usepackage{verbatim}

\newtheorem{theorem}{Theorem}[section]
\newtheorem{lemma}[theorem]{Lemma}
\newtheorem{proposition}[theorem]{Proposition}
\newtheorem{corollary}[theorem]{Corollary}
\theoremstyle{definition}
\newtheorem{definition}[theorem]{Definition}

\newtheorem{example}[theorem]{Example}

\newtheorem{conjecture}[theorem]{Conjecture}
\newtheorem{remark}[theorem]{Remark}

\newcommand{\V}{\mathcal{V}}
\newcommand{\Ext}{\text{Ext}}

\newcommand{\Tr}{\text{Tr}}
\newcommand{\id}{\text{id}}

\newcommand{\Ker}{\text{Ker\,}}

\newcommand{\End}{\text{End}}
\newcommand{\Hom}{\text{Hom}}

\newcommand{\Ind}{\text{Ind}}
\newcommand{\Rep}{\text{Rep}}

\newcommand{\mm}{\mathfrak{m}}

\newcommand{\mR}{{\mathcal R}}

\newcommand{\kk}{{\bold k}}

\newcommand{\C}{\mathcal{C}}
\newcommand{\ot}{\otimes}

\newcommand{\ben}{\begin{enumerate}}
\newcommand{\een}{\end{enumerate}}

\newcommand{\Vect}{{{\bf Vec}}}

\newcommand{\be}{{\bf 1}}

\theoremstyle{plain}

\newtheorem*{sol}{Solution}

\theoremstyle{definition}

\theoremstyle{remark}

\newcommand{\solu}[1]{\begin{sol}{\bf (\ref{#1})}}

\def\C{\mathcal{C}}
\def\D{\mathcal{D}}

\def\Vect{{\bf Vec}}
\def\P{\mathcal{P}}

\def\N{\mathcal{N}}

\def\End{\mathrm{End}}
\def\Hom{\mathrm{Hom}}
\def\Ker{\mathrm{Ker}}

\def\Ext{\mathop{\mathrm{Ext}}\nolimits}

\def\A{\mathcal{A}}

\def\Vec{\mathrm{Vec}}
\def\sVec{\text{sVec}}
\def\Ver{\mathrm{Ver}}
\def\k{\mathbf{k}}

\def\Rep{\mathop{\mathrm{Rep}}\nolimits}

\def\FPdim{\mathop{\mathrm{FPdim}}\nolimits}

\newcommand{\BZ}{{\mathbb Z}}
\newcommand{\BR}{{\mathbb R}}
\newcommand{\cP}{\mathcal{P}}
\newcommand{\Fr}{{\rm Fr}}

\begin{document}

\title{On the Frobenius functor for symmetric tensor categories in positive characteristic}

\author{Pavel Etingof}
\address{Department of Mathematics, Massachusetts Institute of Technology,
Cambridge, MA 02139, USA}
\email{etingof@math.mit.edu}
\author{Victor Ostrik}
\address{Department of Mathematics,
University of Oregon, Eugene, OR 97403, USA}
\address{Laboratory of Algebraic Geometry,
National Research University Higher School of Economics, Moscow, Russia}
\email{vostrik@uoregon.edu}

\begin{abstract} We develop a theory of Frobenius functors for symmetric tensor categories (STC) $\C$ over a field $\bold k$ of characteristic $p$, and give its applications to classification of such categories. Namely, we define a twisted-linear symmetric monoidal functor $F: \mathcal{C}\to \mathcal{C}\boxtimes {\rm Ver}_p$, where ${\rm Ver}_p$ is the Verlinde category (the semisimplification of $\Rep_\k(\mathbb{Z}/p)$); a similar construction of the underlying additive functor appeared independently in \cite{Co}. This generalizes the usual Frobenius twist functor in modular representation theory and also the one defined in \cite{O}, where it is used to show that if $\C$ is finite and semisimple then it admits a fiber functor to ${\rm Ver}_p$. The main new feature is that when $\mathcal{C}$ is not semisimple, $F$ need not be left or right exact, and in fact this lack of exactness is the main obstruction to the existence of a fiber functor $\mathcal{C}\to {\rm Ver}_p$. We show, however, that there is a 6-periodic long exact sequence which is a replacement for the exactness of $F$, and use it to show that for categories with finitely many simple objects $F$ does not increase the Frobenius-Perron dimension. We also define the notion of a Frobenius exact category, which is a STC on which $F$ is exact, and define the canonical maximal Frobenius exact subcategory $\mathcal{C}_{\rm ex}$ inside any STC $\mathcal{C}$ with finitely many simple objects. Namely, this is the subcategory of all objects whose Frobenius-Perron dimension is preserved by $F$. One of our main results is that a finite STC is Frobenius exact if and only if it admits a (necessarily unique) fiber functor to ${\rm Ver}_p$. This is the strongest currently available characteristic $p$ version of Deligne's theorem (stating that a STC of moderate growth in characteristic zero is the representation category of a supergroup). We also show that a sufficiently large power of $F$ lands in $\mathcal{C}_{\rm ex}$. Also, in characteristic 2 we introduce a slightly weaker notion of an almost Frobenius exact category (namely, one having a fiber functor into the category of representations of the triangular Hopf algebra $\bold k[d]/d^2$ with $d$ primitive and $R$-matrix $R=1\otimes 1+d\otimes d$), and show that a STC with Chevalley property is (almost) Frobenius exact. Finally, as a by-product, we resolve Question 2.15 of \cite{EG}.
\end{abstract} 

\maketitle 


\section{Introduction} 

\subsection{} Let $\k$ be an algebraically closed field of characteristic $p>0$. An important role
in the theory of group schemes over $\k$ (and more generally in algebraic geometry over $\k$)
is played by the {\em Frobenius morphism}, see e.g. \cite[I.9]{Ja}. Namely, for any group scheme $G$
over $\k$ we have a homomorphism $F: G\to G^{(1)}$ where $G^{(1)}$ is a suitable twist of $G$,
see {\em loc. cit.} 

Assume now that $G$ is affine.
Taking the pullback, we get the {\em Frobenius twist} functor $\Rep(G^{(1)})\to \Rep(G)$
between the corresponding representation categories, see \cite[I.9.10]{Ja}.
Note that $\Rep(G^{(1)})=\Rep(G)^{(-1)}$ where the only difference between the categories
$\Rep(G)$ and $\Rep(G)^{(-1)}$ is the $\k-$linear structure on the $\Hom-$spaces: the multiplication
by $\lambda$ in $\Rep(G)$ is replaced by the multiplication by $\lambda^{1/p}$ in $\Rep(G)^{(-1)}$.
Equivalently, we can think that the Frobenius twist is an exact symmetric tensor functor 
$\Rep(G)\to \Rep(G)^{(1)}$ where the $\k-$linear structure of the category $\Rep(G)$ is twisted
by the Frobenius automorphism $\lambda \mapsto \lambda^p$ in the category $\Rep(G)^{(1)}$.
The goal of this paper is to introduce and investigate a counterpart of the Frobenius twist functor for symmetric tensor categories more general than $\Rep(G)$.

\subsection{} In the case of a {\em semisimple} symmetric tensor category $\C$ a counterpart
of the Frobenius twist functor was proposed in \cite{O}. A new and somewhat unexpected feature
of the Frobenius functor from \cite{O} is that it takes values in the product $\C^{(1)}\boxtimes \Ver_p$
where $\Ver_p$ is the Verlinde category (the semisimplification of the category $\Rep_\k(\Bbb Z/p)$). Generalizing this construction, in this paper we define a $\k-$linear functor
$F: \C \to \C^{(1)}\boxtimes \Ver_p$ for an abelian symmetric tensor category $\C$ over $\k$ which is not necessarily semisimple, and show that this functor has a natural monoidal structure. 

Unfortunately, in general the functor $F$ is neither left nor right exact. Thus a significant part of this paper is devoted to the study of the behavior of $F$ on exact sequences.

\subsection{} Let us describe some of our results in the special case $p=2$. Thus let
$\C$ be a symmetric tensor category over a field $\k$ of characteristic $2$. 
For an object $X\in \C$ we have
the symmetry $s: X\ot X\to X\ot X$ satisfying $s^2=1$, equivalently $(1-s)^2=0$. We define
$$F(X):=\Ker(1-s)/\mbox{Im}(1-s)$$ 
to be the cohomology of the morphism $1-s$. 
Since $s$ is functorial, we have an obvious structure of a functor on $F$. In this paper we prove
that the functor $F$ enjoys the following properties.

1) $F$ is additive.

2)  $F$ has a natural structure of a monoidal functor.

3) The functor $F$ is not left or right exact, in general. However it has the following property: if
$0\to X \to Y \to Z\to 0$ is a short exact sequence then we have a 3-periodic long exact sequence
$$
\ldots \to F(X)\to F(Y) \to F(Z)\to F(X)\to \ldots
$$

We say that a symmetric tensor category is {\em Frobenius exact} if the functor $F$ is exact. E.g., it is clear that a category of representations of an affine group scheme is Frobenius exact. 
One of our main results is that conversely, any Frobenius exact finite symmetric tensor category is the representation category of a (finite) group scheme, i.e., admits a (necessarily unique) fiber functor to $\Vec_\k$. Moreover, we conjecture that this is true without the finiteness assumption when the category has moderate growth. 

There exist symmetric tensor categories that are not Frobenius exact, e.g. 
the categories $\C_n$ constructed in \cite{BE} in characteristic $2$, or their generalizations to any positive characteristic (\cite{BEO}). For general finite symmetric tensor categories $\C$ (possibly not Frobenius exact) we prove the existence of the maximal Frobenius exact subcategory $\C_{\rm ex}\subset \C$ using the interplay between the Frobenius functor and the Frobenius-Perron dimension. 

Next we proceed to show that if $\C\subset \D$ are finite symmetric tensor categories such that $\C$ contains all simple objects of $\D$ and $\C$ is Frobenius exact then $\D$ is {\it almost Frobenius exact}, i.e., admits a fiber functor to the category $\V$ of representations of the Hopf algebra $\k[d]/d^2$ with symmetry defined by the $R$-matrix  $R=1\otimes 1+d\otimes d$. This generalizes the result of \cite{EG3} 
that any finite symmetric tensor category with Chevalley property is almost Frobenius exact. 

\subsection{} Let us now briefly summarize our results for $p>2$. In this case, similarly to \cite{O}, the Frobenius functor is more complicated, and is a functor $\C\to \C^{(1)}\boxtimes {\rm Ver}_p$, where ${\rm Ver}_p$ is the Verlinde category; it has the form 
$$
F(X)=\oplus_{i=1}^{p-1} F_i(X)\otimes L_i
$$
where $L_i$ are the simple objects of ${\rm Ver}_p$. Nevertheless, it is still additive and monoidal (generalizing properties 1,2 above), and there is also an analog of property 3, a 6-periodic long exact sequence (see Subsection \ref{6per}). This allows us to prove versions of all the above results for $p>2$. In particular, we show that a finite symmetric tensor category $\C$ is Frobenius exact if and only if it admits a (necessarily unique) fiber functor to ${\rm Ver}_p$.
Thus such a category is equivalent to the category of representations of an affine group scheme in ${\rm Ver}_p$ (compatible with the canonical action of $\pi_1({\rm Ver}_p)$). Furthermore, this is expected to hold more generally for categories of moderate growth. 

Moreover, if $\C$ is not necessarily Frobenius exact, we show that it contains a maximal Frobenius exact subcategory $\C_{\rm ex}$, which in the case $p>2$ turns out to be a Serre subcategory (so if $\C\subset \D$ with $\D$ containing the simple objects of $\C$ and $\C$ is Frobenius exact then so is $\D$). This implies that a finite symmetric tensor category with Chevalley property is Frobenius exact, so has a fiber functor to ${\rm Ver}_p$. Finally, for any $p$ we define the notion of the {\it Frobenius order} of $\C$, which is the smallest $n$ such that $F_\bullet^n(\C)\subset \C_{\rm ex}$ is Frobenius exact, and show that it is finite. 

\subsection{} The organization of the paper is as follows. 

Section 2 contains preliminaries. 

Sections 3 and 4 play an auxiliary role. 
Namely, in Section 3 we study additive functors on the category of 
$\k[D]/D^n$-modules (inside an abelian category). This is mostly linear algebra which is essentially well known, and reminiscent of Deligne's study of weight filtrations. 

In Section 4 we develop a theory of monoidal Deligne tensor products of monoidal categories in the case when one of the factors is only Karoubian and not necessarily abelian; this is needed in the study of Frobenius functors. 

In Section 5 we define the main object of study in this paper --- the Frobenius functor, generalizing the definition in \cite{O} to nonsemisimple categories. 

The properties of the Frobenius functor are studied in Section 6. In particular, we show that in categories with finitely many simple objects, the Frobenius functor does not increase the Frobenius-Perron dimension, a property important in applications. 

In Section 7 we define and study Frobenius exact categories, those on which the Frobenius functor is exact, and give several characterizations of this property.  

In Section 8 we state and prove one of our main results -- the characterization of Frobenius exact finite symmetric tensor categories as those admitting a (necessarily unique) fiber functor to the Verlinde category ${\rm Ver}_p$. This is the strongest currently available version of Deligne's theorem (\cite{De2,De3}) in characteristic $p$. Then we derive some corollaries; in particular, we classify finite Frobenius bijective categories. 

In Section 9 we show that for $p>2$, the Serre closure of a Frobenius exact category inside any symmetric tensor category is Frobenius exact, and deduce that any symmetric tensor category with Chevalley property is Frobenius exact (our second main result). We also provide versions of these results in characteristic $2$, using the results of \cite{EG3}. 

Finally, in Section 10 we develop the theory of Frobenius orders of symmetric tensor categories, and show that they are finite. 

The paper also has two appendices containing auxiliary results which, however, seem  interesting in their own right. 

Namely in Appendix A we show that the Hilbert series of a finitely generated commutative algebra in a symmetric tensor category with finitely many simple objects is convergent for $|z|<1$. 

Appendix B shows that an algebra $A$ in a finite tensor category $\C$ is exact if and only if it is simple and there is an embedding ${}^*A\subset A\otimes X$ as $A$-modules for some $X\in \C$. As a by-product, this yields a positive answer to Question 2.15 of \cite{EG}. 

\vskip .05in

{\bf Acknowledgements.} The authors are very grateful to D. Benson, whose discussion with the first author in Spring 2018 triggered this work. They also thank D. Arinkin, K. Coulembier, L. Positselski for useful discussions. The work of P. E. was partially supported by the NSF grant DMS-1502244. 
The work of V.~O. was partially supported by the HSE University Basic Research Program, Russian Academic Excellence Project '5-100' and by the NSF grant DMS-1702251.

\section{Preliminaries} 

\subsection{Conventions} 
Throughout the paper, $\k$ will denote an algebraically closed field, usually of characteristic $p>0$. 
We will freely use the theory of tensor categories and refer the reader to \cite{EGNO} for the basics of this theory. In particular, we will use the conventions of this book. 

Namely, by an {\it artinian}   (or {\it locally finite}, see \cite[1.8]{EGNO})   category over $\k$ we mean a $\k$-linear abelian category 
in which objects have finite length and morphism spaces are finite dimensional, and 
call such a category {\it finite} if it has finitely many (isomorphism classes of) simple objects and enough projectives.  By a  {\it tensor category} we will mean 
the notion defined in \cite{EGNO}, Definition 4.1.1; in particular, such categories are artinian (hence abelian). A semisimple finite tensor category is called a {\it fusion category}. 
We will denote the category of vector spaces by $\Vec_\k$ and for $p\ne 2$ the category of supervector spaces by $\sVec_\k$. If $X$ is an object of a tensor category with finitely many simple objects, then we denote by ${\rm FPdim}(X)$ the Frobenius-Perron dimension of $X$  (\cite{EGNO}, 3.3, 4.5).
Sometimes we will use categories with tensor product which are not abelian but only Karoubian; in this case we will use the term {\it Karoubian monoidal category}. Similarly, by a {\it tensor functor} we will mean exact monoidal functor between tensor categories, see \cite{EGNO}, Definition 4.2.5 (note that such a functor is automatically faithful, see \cite{EGNO}, Remark 4.3.10). In more general situations we will use the term {\it additive monoidal functor}. Also, a symmetric tensor functor out of a symmetric tensor category into another, usually quite concrete one (such as $\Vec_\k$, ${\rm sVec}_\k$ or ${\rm Ver}_p$ defined below) 
will be often called a {\it fiber functor}. 

\subsection{Extension of scalars in categories} Let $\D$ be a locally finite abelian category over $\k$; this means that $\D \simeq C-{\rm comod}$ where $C-{\rm comod}$ is the category of finite dimensional comodules over a $\k-$coalgebra $C$, see \cite[1.9]{EGNO}. Then the category ${\rm Ind}(\D)$ of ind-objects of $\D$ is equivalent
to the category $C-{\rm Comod}$ of possibly infinite dimensional $C-$comodules.
Let $\k \subset K$ be a field extension and let us consider the category ${\rm Ind}(\D)_K$ of ind-objects of $\D$ equipped with a $K-$action. It is clear that this category is equivalent to the category
$C_K-{\rm Comod}$ of comodules over $C_K:=K\otimes_\k C$, i.e., of $C-$comodules over $K$. The category $C_K-{\rm comod}$ of
finite dimensional $C-$comodules over $K$ identifies with the full subcategory $\D_K\subset 
{\rm Ind}(\D)_K$ consisting of objects of finite length. We have an obvious functor of extension of scalars $X\mapsto K\otimes X: {\rm Ind}(\D)\to {\rm Ind}(\D)_K$ and its right adjoint, the functor of restriction of
scalars forgetting the $K-$action.
Since the field $\k$ is algebraically closed, the simple objects of $\D_K$ are in the image of the extension of scalars; in particular the extension of scalars restricts to a functor $\D \to \D_K$.
If the category $\D$ is finite, the same is true for injective and projective objects of $\D_K$.
We will need the following 

\begin{lemma} \label{injectiv}
Assume that the category $\D$ is finite. Assume that the restriction of scalars of
$X\in \D_K$ is injective as an object of category ${\rm Ind}(\D)$. Then $X$ is injective as an object of $\D_K$.
\end{lemma}

\begin{proof} Let $L$ be a simple object of $\D_K$. Then there exists a simple object $\bar L\in \D$
such that $L=K\otimes \bar L$. Since the restriction of scalars functor is adjoint to the 
extension of scalars functor we have an isomorphism of derived functors
$\Ext^1_{\D_K}(K\otimes M,N)=\Ext^1_{{\rm Ind}(\D)}(M,N)$. Hence 
$$\Ext^1_{\D_K}(L,X)=\Ext^1_{\D_K}(K\otimes \bar L,X)=
\Ext^1_{{\rm Ind}(\D)}(\bar L,X)=0.$$
Using the long exact sequence we deduce that $\Ext^1_{\D_K}(M,X)=0$ for any object $M\in \D_K$;
hence $X$ is injective.
\end{proof}

The considerations above apply to the case when
the category $\D$ is a tensor category in the sense of \cite[Definition 4.1.1]{EGNO} (so $\D$ is required to be locally finite); in this case the categories $\D_K$ 
and ${\rm Ind}(\D)_K$ have obvious structures of monoidal categories and the extension of scalars functor has a structure of  a tensor functor. However the functor of tensor product is only right exact a priori and it is clear that the category ${\rm Ind}(\D)_K$ is not rigid. 

\begin{lemma}\label{extsca}
The category $\D_K$ is a tensor category.
\end{lemma}

\begin{proof}
We need to show that the category $\D_K$ is rigid. 
Note that the simple objects of $\D_K$ are rigid since they are in the image of the
extension of scalars.

The results of \cite[Proposition A2]{KL4} or \cite[Theorem A.2.5 (a)]{HPS} imply that 
the rigid objects in $\D_K$ are closed under extensions. Thus we  
deduce by induction in the length that all the objects of $\D_K$ are rigid.
\end{proof}

Another proof of Lemma \ref{extsca} is given in Remark \ref{anoproo}. 

\subsection{The Verlinde category ${\rm Ver}_p$.}
An important role in this paper is played by the {\it Verlinde category} ${\rm Ver}_p={\rm Ver}_p(\k)$ introduced by Gelfand-Kazhdan and Georgiev-Mathieu. So let us recall the basics about this category (see \cite{O,EOV} and references therein). 

The {\it Verlinde category} ${\rm Ver}_{p}$ is a symmetric fusion category over $\kk$
obtained as the quotient of $\Rep_{\kk}({\Bbb Z}/p)$ by the tensor ideal of 
negligible morphisms, i.e. morphisms $f: X\to Y$ such that for any $g: Y\to X$ one has ${\rm Tr}(fg)=0$. 
This category has $p-1$ simple objects, $\be = L_{1}, \ldots, L_{p-1}$, such that
$$
L_{r} \otimes L_{s} \cong \bigoplus_{i=1}^{\min(r, s, p-r, p-s)} L_{|r -s| + 2i - 1}
$$
(the Verlinde fusion rules). We define ${\rm Ver}_{p}^{+}$ to be the abelian subcategory of ${\rm Ver}_p$ generated by $L_{i}$ for $i$ odd, 
and define ${\rm Ver}_{p}^{-}$ to be the abelian subcategory of ${\rm Ver}_p$ generated by $L_{i}$ for $i$ even. By the Verlinde fusion rules, ${\rm Ver}_{p}^{+}$ is a fusion subcategory of ${\rm Ver}_{p}$, 
and tensoring with $\chi:=L_{p-1}$ gives an equivalence of abelian categories ${\rm Ver}_{p}^{+} \rightarrow {\rm Ver}_{p}^{-}$ as long as $p > 2$. Since for $p>2$ the symmetric fusion subcategory generated by $L_{1}$ and $L_{p-1}$ is ${\rm sVec}_\k$, we see that 
$$
{\rm Ver}_{p} = {\rm Ver}_{p}^{+} \oplus {\rm Ver}_{p}^{-} \cong {\rm Ver}_{p}^{+} \boxtimes {\rm sVec}_\k,\ p>2.
$$

The following lemma is standard. 

\begin{lemma} \label{char:Ver} (see e.g. \cite{EOV}) The Frobenius-Perron dimension of $L_r$ is given by the formula
$$
{\rm FPdim}(L_{r}) = \frac{\sin\frac{\pi r}{p}}{\sin\frac{\pi}{p}}.
$$
\end{lemma}

\section{Additive functors on the category of $\k[D]/D^n$-modules}

\subsection{The functors $B_i$ and $E_i$}
Let $n$ be a positive integer and $\A$ an artinian category over a field $\k$ (\cite{EGNO}, 1.8). Let 
$$
\A_n:=\A\boxtimes \Rep (\k[D]/D^n)
$$ 
be the Deligne tensor product (\cite{EGNO}, 1.11), i.e., the category of objects $X$ of $\A$ equipped with an endomorphism $D: X\to X$ such that $D^n=0$. We define the following additive functors $\A_n\to \Vec_\k$: 
$$
B_i(X):=\frac{{\rm Ker} D\cap {\rm Im} D^{i-1}}{{\rm Ker} D\cap {\rm Im} D^i},\ i=1,\dots,n-1
$$
and
$$
E_i(X):={\rm Ker} D^i/{\rm Im} D^{n-i},\ i=1,\dots,n-1.
$$

Also define 
$$
L_{i,j}(X):=\frac{{\rm Ker}D\cap {\rm Im}D^i}{{\rm Ker}D\cap {\rm Im}D^j}, 0\le i<j\le n-1;
$$
It is clear that $L_{i,j}(X)$ has a natural filtration with successive quotients
$B_j(X),\dots,B_{i+1}(X)$ (from bottom to top). 

Finally, for $1\le s\le i\le n-1$ let 
$$
E_{i,s}(X):=\frac{{\rm Ker}D^s\cap {\rm Im}D^{i-s}}{{\rm Im}D^{n-s}}.
$$
It is clear that $E_{i,i}(X)=E_i(X)$ and $E_{i,0}(X)=0$. 

\begin{lemma}\label{le1} (i) For $1\le s\le i$, there is a natural short exact sequence
$$
0\to L_{i-s,n-s}(X)\to E_{i,s}(X)\to E_{i,s-1}(X)\to 0.
$$

(ii) $E_i(X)$ has a natural filtration with successive quotients 
$$
B_{n-i},\dots, B_1; B_{n-i+1},\dots ,B_2;\dots ;B_{n-1},\dots,B_i.
$$ 
In particular, $B_j$ occurs in this filtration with multiplicity
given by the $j$-th entry $v_{ij}$ 
of the vector $\bold v_i:=(1,2,\dots,i,i,\dots,i,\dots,2,1)\in \Bbb Z^{n-1}$ if $i\le n/2$, and 
for $E_{n-i}(X)$ the multiplicities are the same. 
\end{lemma}

\begin{proof} (i) It is easy to see that the operator $D$ defines a surjection $E_{i,s}\to E_{i,s-1}$, whose kernel is 
$L_{i-s,n-s}$. 

(ii) follows from (i) and the composition series for $L_{i,j}$. 
\end{proof} 

\begin{example} Let $n=2$. Then we have just one functor $E_1(X)=B_1(X)={\rm Ker}D/{\rm Im}D$, the cohomology of $D$. 
\end{example}

\begin{proposition}\label{spli} Let $0\to X\to Y\to Z\to 0$ 
be a short exact sequence of finite dimensional $\k[D]/D^n$-modules. 
Suppose that for each $i=1,...,[n/2]$, the sequence 
$0\to E_i(X)\to E_i(Y)\to E_i(Z)\to 0$ 
is exact\footnote{Here and below, $[x]$ denotes the integer part (floor) of a number $x$}. Then the sequence $0\to X\to Y\to Z\to 0$ is split. 
\end{proposition} 

\begin{proof} For $i=1,...,[n/2]$ set $\beta_i(X)$ to be the number of Jordan blocks of size $i$ 
plus the number of Jordan blocks of size $n-i$ in $X$. 
Then the assumption implies that $\beta_i(Y)=\beta_i(X)+\beta_i(Z)$ 
for all $i$. 

  Let $X,Y,Z$ correspond to partitions $\lambda_X$, $\lambda_Y$, $\lambda_Z$ 
encoding the Jordan form of $D$ i.e., the size of the $m$-th largest Jordan block in $X$ is 
given by the $m$-th part of the conjugate partition $\lambda_X^\dagger$, etc (this convention
is dual to the convention in \cite{M}, Chapter 2). Then the existence of the short exact sequence
$0\to X\to Y\to Z\to 0$ implies that the Littlewood-Richardson coefficient $c_{\lambda_X, \lambda_Y}^{\lambda_Z}\ne 0$, see \cite[4.1]{K} (for the reader's convenience we indicate here an argument based
on the basic theory of Hall algebras: for a finite base field the statement follows from \cite[II.4.3 (i)]{M}; in the case of an algebraically closed base field the result follows since the exact sequences of this type can be parameterized by the points of a quasi-affine algebraic variety). 
Considering the partitions $\lambda_X$, $\lambda_Y$, $\lambda_Z$ as dominant $GL_n$ weights
and using the interpretation of the Littlewood-Richardson coefficients as $GL_n$ tensor product 
multiplicities we get that the $GL_n$-weight $\nu:=\lambda_X+\lambda_Z-\lambda_Y\in Q_+$,
where $Q_+$ is the positive part of the root lattice. On the other hand the first
paragraph says that $\nu=\sum_{i=1}^{n-1}\nu_i\alpha_i$ is antisymmetric with respect 
to the flip $\nu_i\mapsto \nu_{n-i}$. Since $\nu_i\ge 0$ we get that $\nu=0$, which implies that the sequence  $0\to X\to Y\to Z\to 0$ splits. 
\end{proof} 

\subsection{The functors $G_i$ for a symmetric tensor category and the 6-periodic long exact sequence}\label{6per} Let $\C$ be a symmetric tensor category over a field $\k$ of characteristic $p$ and let $X\in \C$. 
Let $c=(1...p)$ be the cyclic permutation acting on $X^{\otimes p}$, and let $D=1-c$. 
Then $D^p=1-c^p=0$, so $X^{\otimes p}$ is naturally a $\k[D]/D^p$-module. Define the functors 
$G_i,F_i: \C\to \Vec_\k$ given by 
$$
G_i(X):=E_i(X^{\otimes p}),\ F_i(X):=B_i(X^{\otimes p}),\ i=1,...,p-1.
$$ 

\begin{lemma}\label{addg} The functors $G_i$ and $F_i$ for $i=1,...,p-1$ are additive. 
\end{lemma}

\begin{proof} We have $(X\oplus Y)^{\otimes p}=X^{\otimes p}\oplus Y^{\otimes p}\oplus W$, 
where all summands are $D$-invariant. Moreover, 
$$
W=\oplus_{i=1}^{p-1}X^{\otimes i}\otimes Y^{\otimes p-i}\otimes M_i,
$$ 
where $M_i$ is a free $\k[D]/D^p$-module of $\k$-dimension $\binom{p}{i}$. Thus on $W$ we have
${\rm Ker}D^i={\rm Im}D^{p-i}$, so $E_i(W)=0$ for $i=1,...,p-1$. Hence $B_i(W)=0$ for $i=1,...,p-1$ (as $B_i(W)$ is a composition factor of $E_i(W)$). This implies the statement.  
\end{proof} 

\begin{remark} One can also define the functor $F_p(X)={\rm Im}D^{p-1}$. However, unlike $F_i$ for $i<p$, this functor is not additive, so we will not consider it. 
\end{remark} 

\begin{proposition}\label{longex}
Let $0\to X\to Y\to Z\to 0$ be a short exact sequence in $\C$. Then we have a 6-periodic long exact sequence 
$$
.. \to G_i(X)\to G_i(Y)\to G_i(Z)\to G_{p-i}(X)\to G_{p-i}(Y)\to G_{p-i}(Z)\to ..
$$
In particular, the functor $X\mapsto G_i(X)$ is "exact in the middle" for all $i$. Therefore, the composition series of $G_i(Y)$ is dominated by the composition series of $G_i(X)\oplus G_i(Z)$ (as objects of $\C$). 
\end{proposition} 

\begin{proof} Let $C_i^\bullet(X)$ be the complex with $C_i^j(X)=X^{\otimes p}, j\in \Bbb Z$ and the differential from $C_i^{2k}(X)$ to $C_i^{2k+1}(X)$ given by $D^i$, while the differential from $C_i^{2k-1}(X)$ to $C_i^{2k}(X)$ given by $D^{p-i}$. 
Then we have a short exact sequence of complexes   
\begin{equation} \label{seqofcompl}
0\to C_i^\bullet(X)\to C_i^\bullet(Y)\to \widetilde C_i^\bullet\to 0
\end{equation}
with $\widetilde C_i^\bullet:=C_i^\bullet(Y)/C_i^\bullet(X)$. The natural map
$C_i^\bullet(Y)\to C_i^\bullet(Z)$ induces a surjective map $\widetilde C_i^\bullet\to C_i^\bullet(Z)$. 
We claim that the kernel of the map $\widetilde C_i^l=Y^{\otimes p}/X^{\otimes p}\to C_i^k(Z)=Z^{\otimes p}$ is a free $\k[D]/D^p$-module, so the map $\widetilde C_i^\bullet\to C_i^\bullet(Z)$ is a quasiisomorphism. Indeed the filtration $0\subset X\subset Y$ induces a filtration on the object
$Y^{\otimes p}$ compatible with the action of $D$ and it is sufficient to prove the freeness
for the associated graded. Thus we are reduced to the case $Y=X\oplus Z$ and the claim
follows from the proof of Lemma \ref{addg}.

Thus the cohomology of $\widetilde C_i^\bullet$ identifies with the cohomology of $C_i^\bullet(Z)$ and
the desired 6-periodic sequence is the long exact sequence of cohomology for the short exact sequence
\eqref{seqofcompl}.

\end{proof} 

\begin{example} Let $p=2$. Then we have just one functor $G_1(X)=F_1(X)=F(X)$, which is the cohomology 
of $1-c$ on $X\otimes X$ (called the Frobenius functor below). Thus the $6$-periodic long exact sequence
in this case is $3$-periodic: 
$$
..\to F(X)\to F(Y)\to F(Z)\to..
$$
\end{example} 

\subsection{Subadditive functionals}
Let $\A$ be an artinian category and $\Gamma$ be the split Grothendieck group of $\A$. 
Let $f: \Gamma\to \Bbb R$ be a group homomorphism. 

\begin{definition} We say that $f$ is subadditive if for any short exact sequence 
$0\to X\to Y\to Z\to 0$ we have $f(Y)\le f(X)+f(Z)$. 
\end{definition}

Now let $\A=\C$ be a tensor category with finitely many simple objects. 
For $X\in \C$ let $h_i(X):={\rm FPdim}(G_i(X))$. By Lemma \ref{le1}(ii), 
$$
h_i(X):=\sum_j v_{ij}{\rm FPdim}(F_i(X)).
$$

\begin{corollary}\label{co1} The functionals $h_i$ are subadditive. 
\end{corollary} 

\begin{proof} By Proposition \ref{longex}, the composition series of $G_i(Y)$ is dominated by that of $G_i(X)\oplus G_i(Z)$, which implies the statement. 
\end{proof} 

\begin{corollary}\label{co2} 
Let $\bold a:=(a_1,...,a_{p-1})$ be any sequence of positive numbers such that $a_i=a_{p-i}$ for all $i$, and 
$\bold a$ is concave, i.e., the sequence $a_{m+1}-a_m$ for $m=0,...,p-1$ (with $a_0=a_p:=0$) is decreasing. Then 
the functional $h_\bold a(X):=\sum_j a_j{\rm FPdim}(F_j(X))$ is subadditive. 
\end{corollary} 

\begin{proof} By Corollary \ref{co1}, it suffices to show that $\bold a$ is a nonnegative linear combination of $\bold v_i$ for $1\le i\le p/2$. Write 
$$
\bold a=\sum_{1\le i\le p/2}x_i\bold v_i.
$$
Then we have a system of equations
$$
x_1+...+x_{[p/2]}=a_1,\ x_1+2(x_2+...+x_{[p/2]})=a_2,
$$
and so on. Thus 
$$
x_2+...+x_{[p/2]}=a_2-a_1,
$$
etc. So we get 
$$
x_1=2a_1-a_0-a_2,\ x_2=2a_2-a_1-a_3
$$
and so on. By our assumption, these numbers are nonnegative, which implies the statement. 
\end{proof} 

\begin{example}\label{verl} Let $a_j=\frac{\sin (\pi j/p)}{\sin (\pi/p)}$. 
Since the function $\sin x$ is concave on $[0,\pi]$, the sequence $\bold a$ is concave, so
the functional $h_\bold a$ is subadditive. It is easy to check that in this case 
$$
x_j={\rm tan}\frac{\pi}{2p}\sin \frac{\pi j}{p}.
$$
\end{example}

\section{Functors between Deligne products}

\subsection{Deligne products}
Let $A$ be a finite dimensional algebra over a field $\k$ (of any characteristic) and $\A$ be the category of finite dimensional $A$-modules. Let $\C$ be a $\k$-linear Karoubian category. Recall that a (left) $A$-module in $\C$ is an object $X\in \C$ equipped with an algebra homomorphism $\phi: A\to \End X$ or, equivalently, a morphism $\widehat\phi: A\otimes X\to X$ satisfying associativity. 

\begin{definition} The Deligne tensor product $\A\boxtimes \C$ 
is the category of $A$-modules in $\C$.  
\end{definition} 

It is easy to see that the category $\A\boxtimes \C$ is Karoubian. Also, if $\C$ is abelian then this definition coincides with the usual one (\cite{EGNO}, 1.11). 

Note that any Morita equivalence between algebras $A_1$ and $A_2$ defines a natural equivalence 
$\A_1\boxtimes \C\to \A_2\boxtimes \C$, where $\A_i=A_i-{\rm mod}$. Namely, if 
$A_2=\bold e{\rm Mat}_n(A_1)\bold e$, where $\bold e\in {\rm Mat}_n(A_1)$ is an idempotent then this equivalence is defined by the formula 
$X\mapsto \bold e X^n$, where $X^n:=X^{\oplus n}$. So the notation $\A\boxtimes \C$ is justified.

In fact, the following proposition, communicated to us by D. Arinkin, provides an alternative definition of $\A\boxtimes \C$ for a finite abelian category $\A$ which does not use any choice of an algebra $A$ at all. Namely, let ${\rm Pr}(\A)$ be the full subcategory of projective objects of $\A$, and ${\rm Fun}({\rm Pr}(\A)^{\rm op},\C)$ be the category of additive functors from ${\rm Pr}(\A)^{\rm op}$ to $\C$. 

\begin{proposition}\label{chara} If $\A=A-{\rm mod}$ then there is a natural equivalence 
$$
E: {\rm Fun}({\rm Pr}(\A)^{\rm op},\C)\to \A\boxtimes \C.
$$ 
\end{proposition} 

\begin{proof} 

   By the definition the category $\A\boxtimes \C$ is the category 
of $\k-$linear functors $\underline{A}\to \C$ where $\underline{A}$ is the one-object category with
endomorphism ring $A$. Observe that the category ${\rm Pr}(\A)^{\rm op}$ coincides with the Karoubi
(or pseudo-abelian) envelope of the category $\underline{A}$. The result follows immediately from the universal
property of the Karoubi envelope, see e.g. \cite[1.9]{Desym}.

\end{proof} 

Let $F: \C\to \D$ be an additive ($\k$-linear) functor between Karoubian $\k$-linear categories. If $X\in \C$ is equipped with an action of $A$, then so is $F(X)$. So we have a natural additive functor $F_\A: \A\boxtimes \C\to \A\boxtimes \D$. 

\subsection{Monoidal Deligne products} 
Now suppose that $\A$ is a multiring category (i.e. it has a monoidal structure with biexact tensor product, see \cite{EGNO}, Definition 4.2.3), for example a multitensor category (\cite{EGNO}, Definition 4.1.1).  

\begin{proposition}\label{delprod} (i) If $\C$ is a Karoubian monoidal category with bilinear tensor product then so is $\A\boxtimes \C$. 

(ii) If in addition $\D$ is Karoubian monoidal with bilinear tensor product and $F:\C\to \D$ is an additive monoidal functor then so is $F_\A$. 
\end{proposition}

\begin{proof} The proof is based on realizing $\A$ by linear-algebraic data; a similar approach was used in \cite{DEN} in the proof of Proposition 3.1. Namely, let us write $\A$ as $A-{\rm mod}$ for some finite dimensional algebra $A$. Then the structure of a multiring category on $\A$ defines a (clumsy but sometimes useful) quasibialgebra-like structure on $A$, which one may call a {\it pseudobialgebra}. Specifically, 
the tensor product, since it is exact, defines an $(A,A^{\otimes 2})$-bimodule $T$, projective as a right module, such that $X\otimes Y=T\otimes_{A^{\otimes 2}}(X\boxtimes Y)$. Moreover, we have an associativity isomorphism 
$$\Phi: T\otimes_{A^{\otimes 2}} (T\boxtimes A)\to T\otimes_{A^{\otimes 2}} (A\boxtimes T)$$ 
of $(A,A^{\otimes 3})$-bimodules satisfying the pentagon relation. 

Let us realize $T$ as  $\bold e (A^{\otimes 2})^n$, 
where $\bold e\in {\rm Mat}_n(A^{\otimes 2})$ is an idempotent. Then we have a homomorphism $\Delta: A\to \bold e {\rm Mat}_n(A^{\otimes 2})\bold e$ defining the left $A$-action on $T$. 

Let 
$$
\Delta\otimes 1: A\otimes A\to (\bold e\otimes 1){\rm Mat}_n(A\otimes A\otimes A)(\bold e\otimes 1)
$$  
and 
$$
1\otimes \Delta: A\otimes A\to (1\otimes \bold e){\rm Mat}_n(A\otimes A\otimes A)(1\otimes \bold e)
$$  
be the homomorphisms defined by $\Delta$, 
where we make the identifications 
$$
A\otimes {\rm Mat}_n(A\otimes A)\cong {\rm Mat}_n(A\otimes A\otimes A)\cong {\rm Mat}_n(A\otimes A)\otimes A.
$$
They define homomorphisms 
$$
\Delta_1: {\rm Mat}_n(A\otimes A)\to (\bold e\otimes 1){\rm Mat}_{n^2}(A\otimes A\otimes A)(\bold e\otimes 1)
$$  
and 
$$
\Delta_2: {\rm Mat}_n(A\otimes A)\to (1\otimes \bold e){\rm Mat}_{n^2}(A\otimes A\otimes A)(1\otimes \bold e).
$$  
Let $\bold e_i:=\Delta_i(\bold e)\in {\rm Mat}_{n^2}(A^{\otimes 3})$ be the corresponding idempotents. 
Then we have $T\otimes_{A^{\otimes 2}}(T\boxtimes A)=\bold e_1 (A^{\otimes 3})^{n^2}$, and 
$T\otimes_{A^{\otimes 2}}(A\boxtimes T)=\bold e_2 (A^{\otimes 3})^{n^2}$. 
We have homomorphisms $\eta_i=\Delta_i\circ \Delta: A\to \bold e_i{\rm Mat}_{n^2}(A^{\otimes 3})\bold e_i$
defining the left action of $A$ on these modules.  
Moreover, the associativity isomorphism $\Phi$ is an element of $\bold e_2{\rm Mat}_{n^2}(A^{\otimes 3})\bold e_1$
which commutes with $A$, i.e., $\Phi \eta_1(a)=\eta_2(a)\Phi$, $a\in A$. Finally, we have the distinguished $A$-module $\bold 1$ corresponding 
to the unit object of $\A$. This data (the idempotent $\bold e$, the homomorphism $\Delta$, the module $\bold 1$ and the element $\Phi$) is subject to certain axioms coming from the multiring category structure on $\A$, and may be called a {\it pseudobialgebra} structure on $A$ (the precise axioms are not important for our purposes, so we won't spell them out). 
As we will see in Example \ref{quasibial}, this generalizes Drinfeld's notion of a quasibialgebra. 

Now, the pseudobialgebra structure on $A$ allows us to define a structure of a monoidal category on the category of $A$-modules in any $\k$-linear Karoubian monoidal category $\C$ with bilinear tensor product in the same way as above, i.e., 
by the formula $X\otimes Y:=\bold e (X\boxtimes Y)^n$ (where $X\boxtimes Y$ is the 
tensor product of $X$ and $Y$ in $\C$ regarded as an $A^{\otimes 2}$-module), with the action of $A$ defined by $\Delta$ and associativity defined by $\Phi$. This implies (i).
 
Moreover, it is clear that this structure is respected by additive monoidal functors, which implies (ii). 
\end{proof} 

\begin{remark} The definition of the monoidal structure on $\A\boxtimes \C$ which does 
not use the algebra $A$ is as follows. Since the tensor product on $\A$ is exact, 
it gives rise to an exact functor $T: \A\boxtimes \A\to \A$, where $\A\boxtimes \A$ 
is the usual Deligne tensor product of $\A$ with itself. This functor therefore has a left adjoint functor 
$T^\vee$. Since $T$ is exact, $T^\vee$ maps projectives to projectives, and 
for $F_1,F_2: {\rm Pr}(\A)^{\rm op}\to \C$ we have $(F_1\otimes F_2)(P)=(F_1\boxtimes F_2)(T^\vee(P))$, where 
$F_1\boxtimes F_2: {\rm Pr}(\A\boxtimes \A)^{\rm op}\to \C$ is the external tensor product of $F_1$ and $F_2$ in $\C$ (here we use the definition of the Deligne product provided by Proposition \ref{chara}). 
\end{remark}

\begin{example} \label{quasibial}
Suppose that $\A$ is a finite integral tensor category (\cite{EGNO}, Definition 6.1.13). 
Choose the algebra $A$ in such a way that each simple 
object $X\in \A$ is represented by an $A$-module of dimension ${\rm FPdim}(X)$. Then 
the bimodule $T$ is free of rank $1$ as a right $A^{\otimes 2}$-module, and we may take $n=1$, $\bold e=1$. 
Then $\Delta: A\to A\otimes A$ is a homomorphism (the coproduct), $\Delta_1=\Delta\otimes 1$, $\Delta_2=1\otimes \Delta$, 
and so $\Phi\in A^{\otimes 3}$ is an invertible element such that 
$\Phi (\Delta\otimes 1)\Delta(a)=(1\otimes \Delta)\Delta(a)\Phi$ (the associator). 
Also we have the counit $\varepsilon: A\to \k$ corresponding to the unit object. 
Thus $(A,\Delta,\Phi, \varepsilon)$ is a quasibialgebra, and it has an antipode induced by duality in $\A$, so it is a quasi-Hopf algebra (\cite{EGNO}, 5.12, 5.13). In this case, 
the category $\A\boxtimes \C$ of $A$-modules in $\C$ 
has an obvious structure of a monoidal category, with the usual tensor product 
in $\C$, the $A$-action on $X\otimes Y$ defined by the coproduct $\Delta$ and associativity 
defined by $\Phi$. 
\end{example} 

\begin{remark}\label{anoproo} Another application of the notion of a pseudobialgebra used in the proof of Proposition \ref{delprod} is an alternative proof of Lemma \ref{extsca}. Namely, assume first that $\D$ is finite. Then we can realize $\D$ as the representation category of some pseudobialgebra $A$, and the duality on $\D$ gives rise to an antipode 
on $A$ (similarly to \cite{EGNO}, Definition 5.13.2 in the quasibialgebra case) which gives $A$ a structure one may call a {\it pseudo-Hopf algebra}. Now, it is clear that if $\D$ is the category of finite dimensional modules over a pseudo-Hopf algebra $A$ then $\D_K$ is the category of finite dimensional modules over the pseudo-Hopf algebra $A_K:= K\otimes_\k A$. This makes it evident that $\D_K$ is also a tensor category. 

Finally, the case when $\D$ is only locally finite may be treated similarly, using the dual notion of 
a copseudo-Hopf algebra and considering comodules instead of modules. 
\end{remark} 

\subsection{Monoidal functors between Deligne products arising from the semisimplification functor}
Let $\C$ be a Karoubian pivotal monoidal category over $\k$ satisfying the assumptions of \cite{EOsemi}, Subsection 2.3 (for example, a symmetric tensor category). 
In this case we can define the semisimplification $\overline{\C}$, a semisimple tensor category over $\k$, and we have an additive monoidal semisimplification functor $\bold S: \C\to \overline{\C}$ (\cite{EOsemi}, 2.2, 2.3). Thus by Proposition \ref{delprod}(ii), for any tensor category $\A$ over $\k$ we have the monoidal functor $$\bold S_\A: \A\boxtimes \C\to \A\boxtimes \overline{\C}.$$ If $\A$ is semisimple, this functor coincides with the semisimplification functor of the category $\A\boxtimes \C$. 

Let us regard $\bold S_\A$ just as an additive functor. Then 
$$
\bold S_\A(X)=\oplus_{Y\in {\rm Irr}\overline{\C}} \bold S_\A^Y(X)\otimes Y,
$$ 
where $\bold S_\A^Y: \A\boxtimes \C\to \A$ are additive functors and the set ${\rm Irr}\overline{\C}$ labels the simple objects of $\overline{\C}$. 

Recall that $Y$ may be viewed as an indecomposable object of $\C$ of nonzero dimension. Let $\A=A-{\rm mod}$. 
Then we have 
$$
\bold S_\A^Y(X)=\Hom_\C(Y,X)/\mathcal{N}(Y,X),
$$ 
where $\mathcal{N}(Y,X)$ is the subspace of negligible morphisms. 
It is clear that $\mathcal{N}(Y,X)$ is an $A$-submodule of $\Hom_\C(Y,X)$, so 
$\bold S_\A^Y(X)$ is naturally an $A$-module, i.e., an object of $\A$. 
The expression for $\bold S_\A^Y$ not involving the algebra $A$ is 
$$
\bold S_\A^Y(X)(P)=\Hom_\C(Y,X(P))/\mathcal{N}(Y,X(P)),\quad P\in {\rm Pr}(\A)^{\rm op}. 
$$
(using the definition of the Deligne product provided by Proposition \ref{chara}). 

\begin{example} \label{natfunc}
Let ${\rm char}(\k)=p$, $\C=\Rep_\k(\Bbb Z/p)$, and $\D=\overline{\C}$
its semisimplification -- the Verlinde category ${\rm Ver}_p$ with simple objects $L_1=\be,...,L_{p-1}$. Let $\A$ be a finite abelian category over $\k$. It is clear that $\A\boxtimes \C$ is the equivariantization $\A^{\Bbb Z/p}$ for the trivial action of $\Bbb Z/p$ on $\A$, so 
$\bold S^{L_i}_\A: \A^{\Bbb Z/p}\to \A$. Setting $D:=(1-c)|_X$, where $c$ is the generator of $\Bbb Z/p$, 
we have $\Hom(L_i,X)={\rm Ker}D^i$ and 
$\mathcal{N}(L_i,X)={\rm Ker}D^i\cap {\rm Im}D+{\rm Ker}D^{i-1}$. 
Thus
$$
\bold S_\A^{L_i}(X)=\frac{{\rm Ker}D^i}{{\rm Ker}D^i\cap {\rm Im}D+{\rm Ker}D^{i-1}}.
$$
It is easy to check that the morphism $D^{i-1}$ maps this object isomorphically to 
$$
\frac{{\rm Ker}D\cap {\rm Im}D^{i-1} }{{\rm Ker}D\cap {\rm Im}D^i}=B_i(X).
$$
Thus, we have 
$$
\bold S_\A(X)\cong B(X):=\oplus_{i=1}^{p-1}B_i(X)\otimes L_i, 
$$
which in particular implies 

\begin{proposition}\label{bmon} $B$ is a monoidal functor. 
\end{proposition}

For example, for $p=2$, ${\rm Ver}_p=\Vec_\k$, so we have $B(X)=B_1(X)={\rm Ker}D/{\rm Im} D$, the cohomology of $D$. This is a monoidal functor $\A^{\Bbb Z/2}\to \A$. 
\end{example}

\begin{remark} If the tensor product in $\A$ is only right exact then the functor $\bold S_\A=B: \A^{\Bbb Z/2}\to \A$ does not have to be monoidal. For example, let $\A$ be the category of finite dimensional modules over $R=\k[x]/x^2$ (${\rm char}(\k)=2$), with tensor product over $R$. Let $X=X_0\oplus X_1$, where $X_0=R$ and $X_1=\k$, with $D: X_0\to X_1$ being the augmentation map, and $D|_{X_1}=0$. Also let $Y=\k$. Then $X\otimes Y=\k^2$ with $D(e_1)=e_2$ and $D(e_2)=0$ (where $e_1,e_2$ is the standard basis of $\k^2$), so $B(X\otimes Y)=0$. On the other hand, $B(X)=B(Y)=\k$, so $B(X)\otimes B(Y)=\k$, i.e., $B(X\otimes Y)\ncong B(X)\otimes B(Y)$. 
\end{remark}

\begin{remark} The notion of Deligne tensor product  easily generalizes to the case when $\A$ is an artinian category (i.e., locally finite but not necessarily finite, \cite{EGNO}, 1.8), by replacing the algebra $A$ by the coalgebra $C=A^*$, which can then be taken to be infinite dimensional, and considering $C$-comodules in $\mathcal{C}$. The same applies to the theory of monoidal Deligne products and monoidal functors between them described above. We leave the details to the interested reader. 
\end{remark}

\section{The Frobenius functor}
Let $\C$ be a symmetric tensor category over a field $\k$ of characteristic $p$, and let $\C^{(n)}$ be the $n$-th Frobenius twist of $\C$. The second author introduced\footnote{More precisely, in \cite{O} the functor $F$ is defined in the case when $\C$ is semisimple, but as shown above, the definition makes sense in general.} in \cite{O} a functor 
$F: \C\to \C^{(1)}\boxtimes \Ver_p$ given by 
$$
F(X):=\oplus_{i=1}^{p-1}F_i(X)\otimes L_i\in \C^{(1)}\boxtimes \Ver_p, 
$$
where the functors $F_i$ are defined in Subsection \ref{6per}. 
In other words, we have $F(X)=B(X^{\otimes p})$, where the functor $B$ is defined in Example \ref{natfunc}.

\begin{proposition}\label{addmon} (i) The functor $F$ has a natural structure of an additive monoidal functor. 

(ii) $F$ commutes with symmetric tensor functors. 
\end{proposition} 

\begin{proof} (i) Since the functors $X\mapsto X^{\otimes p}$ and $B$ are monoidal (Proposition \ref{bmon}), so 
is their composition $F$. Also, by Proposition \ref{addg} the functors $F_i$ are additive on objects for all $1\le i\le p-1$. Their additivity on morphisms is proved similarly to \cite{O}, Lemma 3.4.  

(ii) This follows directly from the definition of $F$. 
\end{proof} 

\begin{definition} The functor $F$ is called the Frobenius functor of $\C$. 
\end{definition} 

\begin{corollary} (\cite{Be}, Subsection 8.11) We have 
$$
F_i(X\otimes Y)\cong \oplus_{j=1}^{p-1}\oplus_{s=1}^{{\rm min}(i,j,p-i,p-j)}F_j(X)\otimes F_{|i-j|+2s-1}(Y).
$$
\end{corollary} 

\begin{proof} This follows from Proposition \ref{addmon} and the fusion rules of ${\rm Ver}_p$. 
\end{proof} 

\begin{remark} The functor $F$ appears in \cite{Co}, Section 4 under the name 
{\it external Frobenius twist}, denoted $\Fr$ (see \cite{Co}, Proposition 4.3.3). The functor $F_1$ appears there under the name {\it internal Frobenius twist}, denoted $\Fr_{\rm in}$.  
\end{remark} 

\begin{example} Let $p=2$. Then the Frobenius functor 
$F:\A\to \A^{(1)}$ is given by $F(X)=B(X\otimes X)$, i.e., it is the cohomology of $1-c$ on $X^{\otimes 2}$. 
\end{example} 

\begin{example}\label{FofLi}(\cite{O}) Let $\C={\rm Ver}_p$. Then we have 
$F(L_i)=L_1\boxtimes L_i$ if $i$ is odd and $F(L_i)=L_{p-1}\boxtimes L_{p-i}$ 
if $i$ is even; thus, $F_i=0$ for even $i$. 
By Proposition \ref{addmon}(ii), this implies that 
in any symmetric tensor category $\C$ having a 
symmetric tensor functor $\C\to {\rm Ver}_p$
we have $F_i=0$ for even $i$; e.g., for $p=3$ 
we have $F(X)=F_1(X)\boxtimes \be$. This gives 
a positive answer to Question 4.2.7(ii) of \cite{Co} 
for such categories. 
\end{example} 

\begin{remark} If the tensor product in $\A$ is only right exact, then the Frobenius functor does not have to be monoidal. For example, let ${\rm char}(\k)=2$ and consider the $\Bbb Z$-graded algebra $R=\k[x,y]/(x,y)^m$, where $m$ is sufficiently large. Let $\A$ be the category of $\Bbb Z$-graded $R$-modules, with tensor product over $R$ and $\deg(x)=\deg(y)=1$.  
Let $X\in \A$ be the maximal ideal of $R$ and $Y=\k[x]/x^m$. We are going to show 
that $F(X\otimes Y)_3=\k$, while $(F(X)\otimes F(Y))_3=\k^2$ (where the subscript denotes the homogeneity degree). Since the homogeneous components of  
$F(X\otimes Y)$ and $F(X)\otimes F(Y)$ of any fixed degree $n$ stabilize when $m\to \infty$, we may do the computation for $m=\infty$. 
Then we have $F(X)=(x^2,y^2)$ (the ideal generated by $x^2$ and $y^2$) and $F(Y)=Y=\k[x]$. Also $X\otimes Y=\k[1]\oplus \k[x][1]$, where $[1]$ denotes the degree shift. 
Thus, $F(X\otimes Y)=\k[2]\oplus \k[x][2]$, while $F(X)\otimes F(Y)=\k[x]/x^2[2]\oplus \k[x][2]$. 
Hence $F(X\otimes Y)_3=\k$, while $(F(X)\otimes F(Y))_3=\k^2$ for $m\gg 0$, which gives a desired counterexample. 

Furthermore, if we want $\A$ to be a finite category, we may replace a $\Bbb Z$-grading with a 
$\Bbb Z/N$-grading for large $N$. 
\end{remark} 

\section{Properties of the Frobenius functor}

Let $\C$ be a symmetric tensor category over $\k$ with finitely many simple objects. 

\begin{proposition} \label{subadd} The functional $X\mapsto {\rm FPdim}(F(X))$, $X\in \C$ is subadditive. 
\end{proposition}

\begin{proof} We have ${\rm FPdim}(F(X))=\sum_i a_i{\rm FPdim}(F_i(X))$, where $a_j=\frac{\sin (\pi j/p)}{\sin (\pi/p)}$. 
Thus, the result follows from Corollary \ref{co2} and Example \ref{verl}.
\end{proof} 

\begin{proposition}\label{subadd1} The functional $X\mapsto {\rm FPdim}(F(X))$ is additive on short exact sequences 
(i.e., descends to a character of the Grothendieck ring ${\rm Gr}(\C)$) if and only if the functors $G_i$ are exact for all $i$. 
\end{proposition} 

\begin{proof}    Recall that by the proof of Corollary \ref{co2} we have
$$
{\rm FPdim}(F(X))=\sum_{i=1}^{[p/2]} x_i{\rm FPdim}(G_i(X)),
$$
where $x_i>0$ by Example \ref{verl}. Thus the ``if" direction is clear.
To prove the ``only if" direction, observe that the functionals $X\mapsto {\rm FPdim}(G_i(X))$ are additive on short exact sequences for all $i$ (since they are already known to be subadditive from Corollary \ref{co1}). Since $G_i$ is exact in the middle by Proposition \ref{longex}, this implies that $G_i$ is exact for $i\le p/2$.
If $i>p/2$, the argument is the same, replacing $G_i$ with $G_{p-i}$.   

\end{proof} 

\begin{proposition}\label{lesseq} For any $X\in \C$ one has 
$$
{\rm FPdim}(F(X))\le {\rm FPdim}(X).
$$
\end{proposition}

\begin{proof}
Let $X_i$, $i\in [1,n]$ be the simple objects of $\C$. In the Grothendieck group ${\rm Gr}(\C)$ we have 
$$
X^{\otimes N}=\sum_{i=1}^n c_i(N)X_i,
$$
where $c_i(N)\in \Bbb Z_{\ge 0}$. 
Thus 
$$
{\rm FPdim}(X)^N=\sum_{i=1}^n c_i(N)d_i\ge \sum_{i=1}^n c_i(N), 
$$
where $d_i={\rm FPdim}(X_i)$. We also have 
$$
{\rm FPdim}(F(X))^N={\rm FPdim}(F(X)^{\otimes N})={\rm FPdim}(F(X^{\otimes N}))=
$$
$$
=\sum_{j=1}^{[p/2]}x_j{\rm FPdim}(G_j(X^{\otimes N})). 
$$
Since the functors $G_i$ are exact in the middle, this implies that 
\begin{equation}\label{ineq}
{\rm FPdim}(F(X))^N\le  \sum_{i=1}^n c_i(N)D_i, 
\end{equation}
where $D_i:={\rm FPdim}(F(X_i))=\sum_{j=1}^{[p/2]}x_j{\rm FPdim}(G_j(X_i))$. 
Let $D_{\rm max}=\max_i D_i$. Then we get 
$$
{\rm FPdim}(F(X))^N\le D_{\rm max}{\rm FPdim}(X)^N
$$
for all $N\ge 0$, which implies that ${\rm FPdim}(F(X))\le {\rm FPdim}(X)$. 
\end{proof} 

\begin{theorem}\label{equiva} 
The following conditions on $\C$ are equivalent: 

(i) $F$ maps injections to injections; 

(ii) $F$ maps surjections to surjections; 

(iii) $F$ is exact (i.e., a tensor functor); 

(iv) ${\rm FPdim}(F(X))={\rm FPdim}(X)$ for all $X\in \C$;

(v) The functional $X\mapsto {\rm FPdim}(F(X))$ is additive on short exact sequences; 

(vi) $G_i$ are exact for all $i$;

(vii) $F_1$ maps injections to injections.
\end{theorem} 

\begin{proof}
(i) and (ii) are equivalent by taking duals (since $F$ is monoidal). Also it is clear that (iii) implies (i),(ii) and (vii). 

To show that (i) implies (iii), note that 
if (i),(ii) hold and if 
$$
0\to X\to Y\to Z\to 0
$$
is a short exact sequence, then the sequence 
$$
0\to F(X)\to F(Y)\to F(Z)\to 0
$$
is exact in the first and third term. Denote the cohomology of this sequence in the middle term by $H$. 
Then the composition series of $F(Y)$ is obtained by combining the composition series of $F(X),F(Z)$ and $H$. But by Proposition \ref{subadd}, 
$$
{\rm FPdim}(F(Y))\le {\rm FPdim}(F(X))+{\rm FPdim}(F(Z)).
$$ 
Thus 
${\rm FPdim}(H)=0$, hence $H=0$, as desired. 

To show that (iii) implies (iv), it is enough to note that any tensor functor preserves 
the Frobenius-Perron dimension. 

It is also clear that (iv) implies (v). 

The equivalence of (v) and (vi) is established in Proposition \ref{subadd1}. 

Let us now show that (vi) implies (iii). 
Suppose that $G_i$ are exact for all $i$. Let us realize $\C$ as the category of finite dimensional comodules 
over some coalgebra $C$. 
Consider the short exact sequence 
$$
0\to X^{\otimes p}\to Y^{\otimes p}\to Y^{\otimes p}/X^{\otimes p}\to 0
$$
of $C\otimes (\k\Bbb Z/p)^*$-comodules. Applying the functor $E_i$ 
to this sequence, we get the sequence
$$
0\to G_i(X)\to G_i(Y)\to G_i(Z)\to 0,
$$
which is exact. Forgetting the $C$-coaction 
and using Proposition \ref{spli}, we get that the sequence 
$$
0\to X^{\otimes p}\to Y^{\otimes p}\to Y^{\otimes p}/X^{\otimes p}\to 0
$$
is split as a sequence $\k\Bbb Z/p$-modules. Therefore, the sequence 
$$
0\to B_i(X^{\otimes p})\to B_i(Y^{\otimes p})\to B_i(Y^{\otimes p}/X^{\otimes p})\to 0
$$
splits as a sequence of vector spaces and therefore is exact. But this sequence is nothing but the sequence 
$$
0\to F_i(X)\to F_i(Y)\to F_i(Z)\to 0.
$$
Hence this sequence is exact (although not necessarily split in $\C$, since now we restore the coaction of $C$).
This implies (iii).  

Finally, to show that (vii) implies (i), recall that 
$$
F_1(X\otimes T)=\oplus_{i=1}^{p-1}F_i(X)\otimes F_i(T). 
$$
So if $f: X\to Y$ is an injection then $F_i(f)$ is also an injection for all $i$, since $F_i(f)\otimes {\rm Id}_{F_i(T)}$ 
is a direct summand in $F_1(f\otimes {\rm Id}_T)$, which by assumption is an injection. 
\end{proof} 

\section{Frobenius exact categories}

\begin{definition} We say that a tensor category $\C$ is Frobenius exact if 
the Frobenius functor of $\C$ is exact. 
\end{definition} 

Note that if $\C$ has finitely many simple objects (e.g., if it is finite) then it is Frobenius exact if and only if it satisfies any of the properties (i)-(vii) of 
Theorem \ref{equiva}. 

\begin{remark} By \cite{Co}, Lemma 4.3.2, a symmetric tensor category $\C$ is Frobenius exact if and only if 
it is locally semisimple in the sense of \cite{Co} (i.e., satisfies the equivalent conditions of \cite{Co}, Theorem C). 
\end{remark} 

\begin{remark} There exist finite symmetric tensor categories which are not Frobenius exact. Namely, in characteristic $2$ 
examples are provided by the categories $\C_i$, $i\ge 1$, studied in \cite{BE}. Moreover, in the forthcoming paper \cite{BEO} examples of such categories will be constructed in any positive characteristic. 
\end{remark} 

Let $\C$ be a symmetric tensor category over $\k$ with finitely many simple objects, and let $\C_{\rm ex}\subset \C$ be the full subcategory of objects 
$X$ for which ${\rm FPdim}(F(X))={\rm FPdim}(X)$. 

\begin{proposition}\label{frobex}
The category $\C_{\rm ex}$ is a tensor subcategory of $\C$. It is Frobenius exact, and 
any Frobenius exact tensor subcategory of $\C$ is a tensor subcategory of $\C_{\rm ex}$. 
\end{proposition}  

\begin{proof}
It is clear from condition (iv) of Theorem \ref{equiva} that $\C_{\rm ex}$ is closed under tensor products, duals and taking subquotients, i.e., it is a tensor subcategory. 
Indeed, to see that it is closed under taking subquotients, take a short exact sequence 
$$
0\to X\to Y\to Z\to 0
$$
in $\C$ such that $Y\in \C_{\rm ex}$. Then 
$$
{\rm FPdim}(X)+
{\rm FPdim}(Z)={\rm FPdim}(Y)=
$$
$$
={\rm FPdim}(F(Y))\le {\rm FPdim}(F(X))+{\rm FPdim}(F(Z)) 
$$
(the latter inequality follows from Proposition \ref{subadd}). 
Since by Proposition \ref{lesseq} the Frobenius functor does not increase the Frobenius-Perron dimension, 
it follows that 
${\rm FPdim}(F(X))={\rm FPdim}(X)$ and
${\rm FPdim}(F(Z))={\rm FPdim}(Z)$, i.e., $X,Z\in \C_{\rm ex}$.

The rest is obvious. 
\end{proof} 

\begin{proposition}\label{onlyif}
(i) Let $\C$ admit a symmetric tensor functor to a Frobenius exact category $\D$ (for example, $\D={\rm Ver}_p$). Then $\C$ is Frobenius exact.  

(ii) Let $\C,\D$ be symmetric tensor categories with finitely many simple objects. 
If $E: \C\to \D$ is a symmetric tensor functor then $E(X)\in \D_{\rm ex}$ 
if and only if $X\in \C_{\rm ex}$. 
\end{proposition} 

\begin{proof}
(i) By Proposition \ref{addmon}(ii), the Frobenius functor commutes with symmetric tensor functors, which implies the statement.  

(ii) We have $E(X)\in \D_{\rm ex}$ iff ${\rm FPdim}(E(X))={\rm FPdim}(F(E(X)))$. But $E(F(X))=F(E(X))$ by Proposition \ref{addmon}(ii), and $E$ preserves dimensions, so this is equivalent to   
${\rm FPdim}(X)={\rm FPdim}(F(X))$, which is equivalent to the condition $X\in \C_{\rm ex}$. 
\end{proof} 

\begin{proposition}\label{FPne0} Let $\C$ be a tensor category over $\k$. 
If there exists a projective object not killed by $F$ then 
$\C$ is Frobenius exact. 
\end{proposition}

\begin{proof}    Let $P$ be a projective object not killed by $F$, that is $F(P)\ne 0$. For the
sake of contradiction assume that $\C$ is not Frobenius exact, i.e. there exists a short exact
sequence $S=0\to X\to Y\to Z\to 0$ such that the sequence $F(S)=0\to F(X)\to F(Y)\to F(Z)\to 0$ is not
exact. Then the sequence $F(S)\otimes F(P)$ is also not exact. However this sequence is
isomorphic to the sequence $F(S\otimes P)$ by monoidality of $F$, so it must be exact since the sequence $S\otimes P$ is split and $F$ is additive, see Proposition \ref{addmon} (i). We obtained a contradiction, so the result is proved. 

\end{proof}

\begin{remark} (i) Proposition \ref{FPne0} holds for any additive monoidal functor in the place of $F$, with the same proof. 

(ii) Proposition \ref{FPne0} is not useful if the category $\C$ has no nonzero projective objects.
However, it generalizes straightforwardly to projective pro-objects or injective ind-objects.
\end{remark}

Recall that a tensor category $\C$ is said to be of {\it moderate (or subexponential) growth} if for any $X\in \C$ there exists $C_X\ge 1$ such that the length of $X^{\otimes n}$ is $\le C_X^n$ for all $n\ge 1$ (\cite{EGNO}, 9.11).

\begin{conjecture}\label{con1} Let $\C$ be a symmetric tensor category of moderate growth over a field $\k$ of characteristic $p$. 
Then $\C$ is Frobenius exact if and only if it admits a fiber functor to ${\rm Ver}_p$. 
\end{conjecture} 

Conjecture \ref{con1} is a generalization of the conjecture from \cite{O} that any semisimple symmetric tensor category of moderate growth admits a fiber functor to ${\rm Ver}_p$. 
Note that the ``if" direction of Conjecture \ref{con1} follows from Proposition \ref{onlyif}(i), so just the ``only if" direction requires proof. 

Conjecture \ref{con1} implies that Frobenius exact categories in characteristic 2 are categories of representations of affine group schemes, and in characteristic 3 they are categories of representations of affine supergroup schemes.

In Section \ref{fib} we will prove Conjecture \ref{con1} in the special case of finite tensor categories. 

\begin{remark} The moderate growth condition in Conjecture \ref{con1} cannot be dropped: e.g., the Deligne category ${\rm Rep}(S_t)$, $t\in \Bbb Z_p$ defined by P. Deligne in his letter to the second author 
(see \cite{Ha}) is Frobenius exact (since it is a tensor subcategory in an ultraproduct of Frobenius exact categories), but does not fiber over ${\rm Ver}_p$. 
\end{remark} 

\section{Frobenius exact finite categories fiber over the Verlinde category}\label{fib} 

\subsection{The main theorem} 
In this section we prove Conjecture \ref{con1} in the special case of finite tensor categories
using results of K.~Coulembier \cite{Co}.

\begin{theorem} \label{Kth}
A finite tensor category $\C$ over an algebraically closed field of characteristic $p$ is Frobenius exact if and only if there exists a symmetric tensor (i.e., fiber)
functor $\C \to \Ver_p$. 
Moreover, if this functor exists, it is unique up to a non-unique isomorphism. 
\end{theorem}

Theorem \ref{Kth} is a generalization of \cite{O}, Theorem 1.5, which is a specialization of this theorem to the semisimple case (i.e., when $\C$ is a fusion category). 

\begin{corollary} For $p=2$ finite Frobenius exact categories are precisely the categories of representations of finite group schemes, and for $p=3$ they are the categories of representations of finite supergroup schemes. In particular, for $p=2,3$ any finite Frobenius exact category is integral. 
\end{corollary} 


Theorem \ref{Kth} is proved in the next few subsections. First of all, the uniqueness of the fiber functor if it exists follows from \cite{EOV}, Theorem 2.6; more precisely, this theorem is proved in \cite{EOV} when $\C$ is semisimple, but the proof extends verbatim\footnote{There is one small change required: the existence of the right adjoint functor $I: \Ver_p\to \C$ as in \cite[3.2]{EOV} follows from the exactness of $F: \C \to \Ver_p$ and finiteness of $\C$.} to the non-semisimple case. Thus, it remains to prove the existence of the fiber functor. 

\subsection{Splitting algebras} Let $\Ind(\C)$ be the category of ind-objects of $\C$, see e.g. 
\cite[8.6]{KS}. Let $A\in \Ind(\C)$ be a unital associative algebra. We say that $A$ is a {\em splitting
algebra} if the functor $X\mapsto X\otimes A$ from $\C$ to the category $\Ind(\C)_A$ of right 
$A-$modules splits all short exact sequences in $\C$.   Let $\Vect_\k$ be the category of all (possibly
infinite dimensional) vector spaces over $\k$. 

\begin{lemma} \label{splA}
An algebra $A$ is a splitting algebra if and only if the functor $\Hom(?,A): \C \to \Vect_\k$
is exact.
\end{lemma}

\begin{proof} Assume that $A$ is a splitting algebra. Then the functor $$\Hom_A(X\otimes A,A)=
\Hom(X,A)$$ must be exact, which proves one of the implications.

Conversely, assume that the functor $\Hom(?,A)$ is exact. Then for any $X\in \C$ the functor
$\Hom(?,X\ot A)=\Hom(X^*\ot ?,A)$ is exact. Hence the functor $\Hom_A(?\ot A,X\ot A)$ is exact.
Thus for a short exact sequence $0\to X\xrightarrow{f} Y\to Z\to 0$ the sequence
$$0\to \Hom_A(Z\ot A,X\ot A)\to \Hom_A(Y\ot A,X\ot A)\to \Hom_A(X\ot A,X\ot A)\to 0$$
is exact. Hence there exists $\phi \in \Hom_A(Y\ot A,X\ot A)$ such that 
$$\phi \circ (f\ot \id_A)=\id_{X\ot A}.$$
Thus $\phi$ is a splitting of the short exact sequence of $A-$modules
$$0\to X\ot A\xrightarrow{f\ot \id_A}Y\ot A\to Z\ot A\to 0.$$
\end{proof} 

\begin{remark} The ind-objects $A$ such that the functor $\Hom(?,A)$ is exact are called 
{\em quasi-injective} in \cite[15.2]{KS}.
It was pointed out to us by Leonid Positselski that for a locally finite (or, more generally, a
Noetherian) category $\C$ quasi-injective objects are precisely the injective objects of the category
$\Ind(\C)$.
\end{remark}

The following theorem follows from \cite[Theorem C, Lemma 4.3.2]{Co}. 

\begin{theorem}[\cite{Co}] \label{locsemi}
Let $\C$ be a Frobenius exact symmetric tensor category. Then there exists
a nonzero commutative splitting algebra in $\Ind(\C)$.
\end{theorem} 

\subsection{$\mm-$negligible morphisms} \label{mnegl}
Let $A\in \Ind(\C)$ be a commutative algebra. 
We have a symmetric monoidal functor $X\mapsto X\ot A$ from $\C$ to the category of $A-$modules
in $\Ind(\C)$. Clearly the image of this functor consists of rigid objects.
We consider the category $\P_A$ of $A-$modules in $\Ind(\C)$ which are direct summands of 
$A-$modules of the form $X\ot A$, $X\in \C$. This is a rigid symmetric monoidal category, 
nonabelian in general. Thus the traces of morphisms are defined in $\P_A$ and take values
in $A_0:=\Hom_A(A,A)=\Hom(\be,A)$ (so $A_0$ is a usual commutative algebra in $\Vect_\k$). 

Pick a maximal ideal $\mm\subset A_0$ and say that a morphism $f: X\to Y$ in $\P_A$ is
$\mm-${\it negligible} if the trace $\Tr(fg)\in \mm$ for any $g:Y\to X$. It is easy to see that 
$\mm-$negligible morphisms form a tensor ideal in the category $\P_A$, see e.g. \cite{EOsemi}. 
Let $\P_A^\mm$ be the quotient by this ideal, see {\em loc. cit.} The category $\P_A^\mm$ is a rigid
symmetric monoidal category. The endomorphism algebra of the unit object in $\P_A^\mm$ is
the field $K=A_0/\mm$. 

\begin{definition} We call a rigid Karoubian symmetric monoidal category $\P$ {\it non-degenerate} if for any nonzero morphism $f:X\to Y$ there is a morphism $g:Y\to X$ such that $\Tr(fg)\ne 0$.
\end{definition} 

It is easy to see that the category $\P_A^\mm$ is non-degenerate. 
Also the composition of the projection $\P_A\to \P_A^\mm$ and of the functor $X\mapsto X\ot A$ is 
a symmetric monoidal functor $\bold F: \C \to \P_A^\mm$.
Finally, the functor $\bold F$ is surjective in a sense that any object of $\P_A^\mm$ is a direct summand of an object $\bold F(X)$, $X\in \C$ (cf. \cite{EGNO}, Definition 1.8.3).

Now assume that $A$ is a commutative splitting algebra. Then by definition the functor $\C \to \P_A$
sends any short exact sequence to a split one; hence the functor $\bold F: \C\to \P_A^\mm$ is exact (in the sense that it sends short exact sequences in $\C$ to split short exact sequences in $\P_A^\mm$). Thus we have the following consequence of Theorem \ref{locsemi}:

\begin{corollary} \label{quot}
Let $\C$ be a Frobenius exact symmetric tensor category. Then there exists
a field extension $\k \subset K$, a $K-$linear non-degenerate symmetric monoidal category 
$\mR$ with $\End(\be)=K$, and an exact
$\k-$linear surjective symmetric monoidal functor $\bold F:\C \to \mR$.\footnote{Note that by further extending scalars if needed, we may assume that $K$ is algebraically closed.}
\end{corollary}

\subsection{Finiteness conditions} \label{scalars}
Recall that non-degenerate categories are not abelian in general (they may contain nilpotent endomorphisms with nonzero trace), see \cite[5.8]{Desym}. However we will show that if $\C$ is finite then the category $\mR$ in Corollary \ref{quot} can be chosen abelian and semisimple. Moreover, we will show that we can assume that $K=\k$ and that all $\Hom-$spaces in the category $\mR$
are finite dimensional over $\k$. We note that this is not automatic, and refer the reader to \cite[2.19]{De2} for an example of a non-degenerate category with infinite dimensional $\Hom-$spaces.

Now let $\C, K, \mR, F$ be as in Corollary \ref{quot}. Since the functor 
$$
X\mapsto \Hom_\mR(F(X),\be)
$$ 
is exact, 
it is representable by an ind-object $B\in \Ind(\C)$. Thus we have a functorial isomorphism
\begin{equation}\label{Brep}
\Hom_\mR(F(X),\be)\simeq \Hom(X,B).
\end{equation}
It follows that $B$ has a natural structure of a commutative algebra (cf. \cite[Lemma 3.5]{DMNO}). 
Also $\Hom(\be,B)=\Hom_\mR(\be,\be)=K$. Thus $B$ is equipped with a $K-$action. In other words, 
$B$ can be regarded as an object of the category ${\rm Ind}(\C)_K$.

Now assume that $\C$ is finite. 
Then Corollary \ref{c3} shows that 
$B$ has finite length as an object of ${\rm Ind}(\C_K)$, i.e., belongs to $\C_K$. 
Also $B$ is injective as an object of $\Ind(\C)$, hence injective as an object of $\C_K$ by
Lemma \ref{injectiv}. Thus there exists an injective object $\overline{B}\in \C$ such
that $K\otimes\overline{B}=B\in \C_K$ has a structure of a commutative algebra. Since $\k$ is algebraically
closed we see that $\overline{B}$ itself has a structure of a commutative algebra (such structures are described
by solutions of some polynomial equations with finitely many variables, so the claim follows from the Nullstellensatz). By Lemma \ref{splA} it follows
that $A=\overline{B}$ is a nonzero commutative splitting algebra in $\C$ (as opposed to $\Ind(\C)$). 
For such an
algebra $A$, we have that $A_0=\Hom(\be,A)$ is finite dimensional over $\k$; hence $A_0/\mm=\k$
for any maximal ideal $\mm\subset A_0$. Also all $\Hom-$spaces between $A-$modules of
the form $X\otimes A,\; X\in \C$ are finite dimensional over $\k$. Thus arguing as in Section \ref{mnegl} we get the following improvement of Corollary \ref{quot}:

\begin{corollary} \label{bquot}
Let $\C$ be a finite Frobenius exact symmetric tensor category. Then there exists
a $\k-$linear non-degenerate symmetric monoidal category 
$\mR$ with finite dimensional $\Hom-$spaces, $\End(\be)=\k$, and an exact
$\k-$linear surjective symmetric monoidal functor $\bold F:\C \to \mR$.
\end{corollary}

\subsection{Semisimplicity of $\mR$} Our goal now is to show that the category $\mathcal{R}$ is abelian and moreover semisimple. 

Let $\C, \mR, \bold F$ be as in Corollary \ref{bquot}. Let $B\in \C$ be a 
commutative algebra satisfying \eqref{Brep}. 

\begin{lemma} \label{Bsim}
{\em (i)} The category $\mR$ is monoidally equivalent to a full subcategory of
the category $\C_B$ of $B-$modules.

{\em (ii)} The algebra $B$ is simple;

{\em (iii)} There exists an invertible object $\delta \in \C$ and an isomorphism of
$B-$modules $B^*\simeq \delta \ot B$.
\end{lemma}

\begin{proof} (i) Let $I: \mR \to \C$ be the right adjoint functor to $\bold F$ (it exists as the
functor $\Hom_\mR(\bold F(?),M)$ is exact, hence representable by an object $I(M)\in \C$). 
It is easy to see that $I(M)$ is an $I(\be)=B-$module in a natural way, cf. \cite[3.2]{DMNO}.
Thus the functor $I$ upgrades to a functor $I:\mR \to \C_B$. Moreover, for 
$M=\bold F(X)$ we have 
$$\Hom_\mR(\bold F(?),\bold F(X))=\Hom_\mR(\bold F(?)\ot \bold F(X)^*,\be)=\Hom_\mR(\bold F(?\ot X^*),\be)$$
$$=\Hom(?\ot X^*,B)=\Hom(?,X\ot B),$$
so $I(\bold F(X))$ is the free module $X\ot B$. Also
$$\Hom_B(X\ot B,Y\ot B)=\Hom(X,Y\ot B)=\Hom(X\ot Y^*,B)$$
$$=\Hom_\mR(\bold F(X\ot Y^*),\be)=
\Hom_\mR(\bold F(X),\bold F(Y))$$
so the functor $I$ is fully faithful on objects of the form $\bold F(X)$. Since $\bold F$ is surjective, we get
that $I$ is fully faithful. Similarly, the functor $I$ has an obvious monoidal structure on the
objects of the form $\bold F(X)$, whence we get the monoidal structure on $I$.

(ii) Assume $J\subset B$ is an ideal. Then the multiplication $J\ot B\to B$ is a $B-$module map,
hence it corresponds to a nonzero morphism $F(J)\to \be$ in the category $\mR$ by (i). 
Since $\mR$ is non-degenerate, we can find a morphism $\be\to F(J)$ such that the compostion
$\be\to F(J)\to \be$ is not zero. Using (i) we should have a $B-$module morphism $B\to J\ot B$
such that the composition $B\to J\ot B\to B$ is a nonzero multiple of the identity morphism.
But this is impossible since the image of multiplication $J\ot B\to B$ is $J\ne B$.

(iii) Recall that $B$ is an injective object of $\C$ containing a copy of the injective hull 
of the unit object. The cosocle of this injective hull is an invertible object of $\C$, see
\cite[2.8]{EO}. Thus there is an invertible object $\delta$ and nonzero morphism 
$\delta \to B^*$. Thus we have a nonzero morphism of $B-$modules $\delta \ot B\to B^*$.
This is an embedding as $B$ is a simple $B-$module by (ii). Thus it must be an isomorphism
as the Frobenius-Perron dimensions of both modules are the same. 
\end{proof} 

\begin{corollary} {\em (i)} The category $\C_B$ of $B-$modules in $\C$ is rigid.

{\em (ii)} The trace of a nilpotent endomorphism in $\mR$ is zero.

{\em (iii)} The category $\mR$ is semisimple.
\end{corollary}

\begin{proof} (i) Lemma \ref{Bsim} (ii), (iii) and Theorem \ref{NZ} imply that the algebra
$B$ is exact. Hence any $B-$bimodule is rigid by \cite[3.3]{EO}. Any $B-$module can be
regarded as a $B-$bimodule with the same left and right actions. Also any $B-$module $M$
is a quotient of a free $B-$module of the form $X\ot B$, e.g. for $X=M$. Thus the dual $B-$bimodule $M^*$
is a subbimodule of $(X\ot B)^*=X^*\ot B$. Thus the left and right actions of $B$ on $M^*$
coincide. Thus $M^*$ regarded as a right $B-$module is the dual of $M$ in the category of
$B-$modules. 

(ii) The category $\C_B$ is an abelian rigid symmetric tensor category. Thus the trace of
a nilpotent endomorphism in $\C_B$ is zero, see e.g. \cite[3.6]{Desym}. Hence this applies to $\mR$
in view of Lemma \ref{Bsim}(i). 

(iii) The endomorphism algebra of any object in $\mR$ has no nonzero nilpotent ideals by (ii) and
non-degeneracy of $\mR$. Thus this endomorphism algebra is semisimple. The result follows.

\end{proof} 

\subsection{Proof of Theorem \ref{Kth}} We have already constructed a surjective exact 
symmetric tensor functor
$F:\C \to \mR$ where $\mR$ is a semisimple symmetric tensor category. 
Then the category
$\mR$ is a fusion category as all of its simple objects are contained in $F(P)$ where $P\in \C$ is a
projective generator. Thus by \cite{O}, Theorem 1.5 there exists a symmetric tensor
functor $\mR \to {\rm Ver}_p$. The result follows.

\subsection{Frobenius bijective categories} 

Let $\C$ be a Frobenius exact category. Let $F_\bullet(\C)\subset \C$ be the full subcategory of $\C$ consisting of subquotients of direct sums of objects of the form $F_i(X)$, 
$X\in \C$, $1\le i\le p-1$. It is clear that $F_\bullet(\C)$ is a tensor subcategory of $\C$. Note that we have the extended Frobenius tensor functor 
$$
\widehat{F}: 
\C\boxtimes {\rm Ver}_p\to F_\bullet(\C)^{(1)}\boxtimes {\rm Ver}_p
$$
acting identically on the second component.

\begin{definition}
A Frobenius exact category $\C$ is called Frobenius bijective if $F_\bullet(\C)=\C$ and 
$\widehat{F}$ is an equivalence.  
\end{definition}

It is clear that $\widehat{F}$ is a surjective tensor functor. Thus a finite Frobenius exact category $\C$ is Frobenius bijective if and only if $F_\bullet(\C)=\C$ (i.e., $\widehat F$ is an equivalence automatically). 


\begin{corollary}\label{lands} (i) Let $G$ be a finite supergroup scheme with purely odd Lie algebra 
${\rm Lie}(G)$ and a central element $z\in G_0$ of order $\le 2$ acting by sign on 
${\rm Lie}(G)$; if $p=2$ then we assume that $G$ is just a finite group and $z=1$. 
Then the category $\Rep(G,z)$ is Frobenius bijective.

(ii) Every Frobenius bijective finite tensor category is of the form (i). 
\end{corollary} 

\begin{proof} (i) Let $G$ be a finite supergroup scheme and $z\in G_0$ an element of order $\le 2$ acting by parity on $O(G)$ (for $p=2$, just a finite group scheme, with $z=1$). Then the Frobenius functor on the category $\Rep(G,z)$ is just induced by the Frobenius map ${\rm Fr}: G\to G$. Hence 
$\Rep(G,z)$ is Frobenius bijective if and only if ${\rm Fr}$ is an isomorphism. 
For this, the Frobenius map ${\rm Fr}: G_0\to G_0$ must be an isomorphism, which implies that 
$G_0$ is \'etale, i.e., a finite group. Conversely, it is clear that if $G_0$ is a finite group then ${\rm Fr}: G\to G$ is an isomorphism, hence $\Rep(G,z)$ is Frobenius bijective. 

(ii) By Theorem \ref{Kth}, there exists a fiber functor $E: \C\to {\rm Ver}_p$. 
It was shown in \cite{O} that the functors $F_i$ on ${\rm Ver}_p$ land in $\Vec_\k$ or $\sVec_\k$ (see Example \ref{FofLi}). This means that $F_i(E(X))$ is a (super)space for all $X$. But $F_i(E(X))=E(F_i(X))$.
Since every object of $\C$ is a subquotient of a direct sum of $F_i(X)$, we get that $E(Y)$ is a (super)space for all $Y\in \C$. Thus $\C\cong {\rm Rep}(G,z)$, for a finite supergroup scheme $G$ and an element $z\in G_0$ of order $\le 2$ acting by parity on $O(G)$ (where for $p=2$, $G$ is a finite group scheme and $z=1$). Moreover, by the proof of (i), Frobenius bijectivity implies that $G_0$ is a finite group, i.e., ${\rm Lie}(G)$ is purely odd, as claimed. 
\end{proof} 

In particular, it follows that every Frobenius bijective    finite tensor   category is integral. 

\subsection{Almost Frobenius exact categories in characteristic $2$.} 
Consider the case $p=2$. Then it is useful to slightly relax the condition of Frobenius exactness. 
Namely, let $\V$ be the category of representations of the Hopf algebra $\k[d]/d^2$ with $d$ primitive and symmetry defined by the triangular \linebreak $R$-matrix $R=1\otimes 1+d\otimes d$; this category was described in \cite{V}, Subsection 1.5. The category $\V$ has just one simple object $\bold 1$ and another indecomposable $P$, the projective cover of $\be$, which is the regular representation of $\k[d]/d^2$, with $F(\be)=\be$ and $F(P)=0$. 
Thus $F: \V\to \Vec_\k$. 

The following definition is motivated by Theorem \ref{Kth}. 

\begin{definition} A symmetric tensor category $\C$ is said to be {\it almost Frobenius exact} if 
it admits a fiber functor to $\V$. 
\end{definition} 

\begin{proposition}\label{alm} If $\C$ is almost Frobenius exact    and has finitely many simple objects   then $F(\C)\subset \C_{\rm ex}$. 
\end{proposition} 

\begin{proof} Let $E: \C\to \V$ be a fiber functor. Given $X\in \C$, we have $E(F(X))=F(E(X))\in \Vec_\k$. 
Thus by Proposition \ref{onlyif}(ii), $F(X)\in \C_{\rm ex}$, as claimed. 
\end{proof}

\section{Categories with Chevalley property}

\begin{theorem}\label{subca} Let $\D$ be a finite symmetric tensor category  over a field $\k$ of characteristic $p>2$, and $\C\subset \D$ a tensor subcategory containing all the simple
objects of $\D$. Then if $\C$ has a fiber functor to ${\rm Ver}_p$ then so does
$\D$.
\end{theorem} 

\begin{proof} We will use the following lemma. Let $R$ be the $p-2$-dimensional
irreducible representation of $S_p$ on $\k$-valued functions on $[1,p]$ with zero sum
of values modulo constants.

\begin{lemma}\label{lemm1} For any $m\ge 1$ we have 
$L_m^{\otimes p}=\oplus_{j=1}^{p-1} L_j\otimes M_j$, $M_j\in {\rm Rep}_\k(S_p)$, where 
$M_j$ is projective unless $j=1$ for odd $m$ and $j=p-1$ 
for even $m$. Moreover, for this particular $j$ we have 
$M_j\cong \wedge^{m-1}R\oplus P$, where $P$ is a projective $S_p$-module. 
\end{lemma} 

\begin{proof} Consider the Frobenius functor 
$$
F: {\rm Ver}_p\to {\rm Ver}_p\boxtimes {\rm Ver}_p.
$$ 
Let $\chi:=L_{p-1}$. By Example \ref{FofLi}, that 
$F(L_m)=\bold 1\boxtimes L_m$ when $m$ is odd 
and $F(L_m)=\chi\boxtimes L_{p-m}$ when $m$ is even. 
This implies that the object 
$L_m^{\otimes p}\in {\rm Ver}_p\boxtimes \Rep_\k(\Bbb Z/p)$ 
has the form $\oplus_{j=1}^{p-1} L_j\otimes M_j$, where 
$M_j$ is a free $\k\Bbb Z/p$-module unless $j=1$ for odd $m$ and $j=p-1$ 
for even $m$. But if a finite dimensional $S_p$-module $M$ is free 
over $\k\Bbb Z/p$ then it is projective, since the trivial $S_p$-representation
$\bold 1$  is a direct summand in ${\rm Ind}_{\Bbb Z/p}^{S_p}\bold 1$. 
Thus $M_j$ are projective $S_p$-modules unless 
$j=1$ for odd $m$ and $j=p-1$ for even $m$; moreover, 
for this particular $j$, $M_j$ has a unique indecomposable 
summand which is not projective over $S_p$. 

It remains to compute this summand. By \cite{E}, Lemma 6.3, 
$$
L_2^{\otimes p}=\chi\otimes R\oplus \bigoplus_{i=1}^{p-1}L_i\otimes Q_i,
$$ 
where $Q_i$ are projective $S_p$-modules. This implies that 
$$
S^{m-1}(L_2^{\otimes p})=S^{m-1}(\chi\otimes R)\oplus \bigoplus_{i=1}^{p-1}L_i\otimes Q_i'=\chi^{\otimes m}\otimes \wedge^{m-1}R\oplus \bigoplus_{i=1}^{p-1}L_i\otimes Q_i',
$$ 
where $Q_i'$ are projective $S_p$-modules. But $L_m=S^{m-1}L_2$, 
hence $L_m^{\otimes p}$ is a direct summand in $S^{m-1}(L_2^{\otimes p})$. 
Thus the only non-projective summand in $L_m^{\otimes p}$ must be $\chi^{\otimes m}\otimes \wedge^{m-1}R$, as claimed. 
\end{proof} 

\begin{lemma}\label{lemm2} Let $X\in {\rm Ver}_p$, $p>2$. Then $H^1(S_p,\Hom(\be,X^{\otimes p}))=0$. \end{lemma} 

\begin{proof} Representations of $S_n$ with $n<p$ over $\k$ are
semisimple. Therefore, by the Shapiro lemma, it is sufficient to prove
the lemma for simple $X$, and we may assume that $X=L_{2r+1}\in {\rm Ver}_p^+$ for some $r\ge 0$, otherwise $\Hom(\be,X^{\otimes p})=0$. By Lemma \ref{lemm1}, 
$\Hom(\be,X^{\otimes p})=\wedge^{2r} R\oplus P$, where $P$ is a projective $S_p$-module. 
So it suffices to show that $H^1(S_p,\wedge^{2r}R)=0$. 

However, it is easy to see that $H^1(S_p,\wedge^j R)=0$ for all $j\ne 1$. Indeed, the category ${\rm Rep}_\k(S_p)$ has a unique non-semisimple block of defect $1$ with simples $V_j:=\wedge^jR$, $j=0,...,p-2$, and 
the projective cover of $V_j$ can involve only $V_j,V_{j+1},V_{j-1}$ as composition factors (\cite{JK}, Chapter 6), which 
implies the claim. 

Thus,  $H^1(S_p,\wedge^{2r}R)=0$, which proves Lemma \ref{lemm2}.  
\end{proof} 

Now we proceed with the proof of Theorem \ref{subca}. We will show that $\D$ is
Frobenius exact. By Theorem \ref{Kth}, this will imply that $\D$
fibers over ${\rm Ver}_p$.

By \cite{Co}, Theorem C (the implication ``(iii) implies (i)"), it suffices to show that for each $V\in \D$ the
surjective map $f: S^p({\rm gr}V)\to {\rm gr} S^pV$ is an isomorphism, where the
filtration on $V$ is the Loewy filtration. Note that ${\rm gr}V\in \C$. Let $E:
\C\to {\rm Ver}_p$ be a fiber functor. Let $E': \D\to {\rm Ver}_p$ be a
quasi-tensor lift of $E$ (\cite{EGNO}, Definition 4.2.5); it is not hard to show that it exists,   see \cite[Proposition 7.5]{EG2}. 
Then we get an $S_p$-action on $E'(V)^{\otimes p}$ (which may not be the tautological one since $E'$
is only quasi-tensor and not tensor in general), whose associated graded is the tautological $S_p$-action on
$E({\rm gr}V)^{\otimes p}$. By Lemma \ref{lemm2}, $H^1(S_p,\End_\k(E({\rm gr}V)^{\otimes p}))=0$ 
(namely, we take $X=E({\rm gr}V)\otimes E({\rm gr}V)^*$). This implies that there is a filtered
isomorphism of $S_p$-modules $E({\rm gr}V)^{\otimes p}\to E'(V)^{\otimes p}$ whose associated graded
is the identity. Thus we have an isomorphism of coinvariants
$E({\rm gr}V)^{\otimes p}_{S_p}=E(S^p({\rm gr}V))\to E'(V)^{\otimes p}_{S_p}$. But
$E'(V)^{\otimes p}_{S_p}=E'(S^pV)$. So we have a filtered isomorphism
$E(S^p({\rm gr}V))\to E'(S^pV)$. Its associated graded is the map $E(f)$. 
Note that $E(f)$ is surjective, since so is $f$ and $E$ is exact. Thus
$E(f)$ is an isomorphism (since it is a surjective morphism 
between objects of the same Frobenius-Perron dimension). 
Since $E$ is faithful, this implies that $f$ is an
isomorphism as well. This proves Theorem \ref{subca}. 
\end{proof} 

Recall (\cite{EGNO}, Subsection 4.12) that a tensor category $\D$ is said to have {\it Chevalley property} 
if the tensor product of any two simple objects of $\D$ is semisimple. 

\begin{corollary}\label{chev} If $\D$ is a finite symmetric tensor category with Chevalley property over a field $\k$ of characteristic $p>2$ then $\D$ has a fiber functor to ${\rm Ver}_p$.
\end{corollary} 

\begin{proof} Take $\C$ to be the semisimple part of $\D$. Then $\C$ is a
fusion category, so by \cite{O}, Theorem 1.5 it fibers over ${\rm Ver}_p$. Thus the result follows from  
Theorem \ref{subca}. 
\end{proof} 

\begin{corollary}\label{serresubcat} If $p>2$ then the category $\C_{\rm ex}$ is a Serre subcategory of $\C$. 
\end{corollary}  

\begin{proof} Let $\widetilde{\C}_{\rm ex}$ be the Serre closure of $\C_{\rm ex}$. 
This category has the same simple objects as $\C_{\rm ex}$. Also by Theorem \ref{Kth} there 
is a fiber functor $\C_{\rm ex}\to {\rm Ver}_p$. Thus by Theorem \ref{subca}, 
there is also a fiber functor $\widetilde{\C}_{\rm ex}\to {\rm Ver}_p$, i.e., 
$\widetilde{\C}_{\rm ex}$ is Frobenius exact. Hence $\widetilde{\C}_{\rm ex}=\C_{\rm ex}$, i.e., 
$\C_{\rm ex}$ is a Serre subcategory, as claimed. 
\end{proof} 

For $p=2$ Corollary \ref{chev} and hence Theorem \ref{subca} fails: a counterexample is the category $\D=\V$ and $\C=\Vec_\k$. However, the following theorem gives a correct version of Theorem \ref{subca} in characteristic $2$. 

\begin{theorem}\label{subca2} (\cite{EG3}, Theorem 2.21(2)) Let $\D$ be a finite symmetric tensor category  over a field $\k$ of characteristic $2$, and $\C\subset \D$ a tensor subcategory containing all the simple
objects of $\D$. Then if $\C$ has a fiber functor to $\Vec_\k$ (i.e., is Frobenius exact) then $\D$ has a fiber functor to $\V$ (i.e., is almost Frobenius exact). 
\end{theorem} 

\begin{corollary} (\cite{EG3}, Theorem 1.1) Any symmetric tensor category with Chevalley property 
over a field $\k$ of characteristic $2$ is almost Frobenius exact. 
\end{corollary} 

\begin{corollary}\label{cha2} Let $p=2$ and $\widetilde{\C}_{\rm ex}$ be the Serre closure of $\C_{\rm ex}$. Then $\widetilde{\C}_{\rm ex}$ is almost Frobenius exact. In particular, $F(\widetilde{\C}_{\rm ex})\subset \C_{\rm ex}$. 
\end{corollary} 

\begin{proof} By Theorem \ref{Kth}, there exists a fiber functor   $\C_{\rm ex}\to \Vec_\k$. Thus by Theorem \ref{subca2}, there is a fiber functor $E: \widetilde{\C}_{\rm ex}\to \V$, giving the first statement.
The second statement follows from the first one and Proposition \ref{alm}.  
\end{proof} 

\begin{remark} In the case when $\D$ is an integral category (i.e., all objects have integer Frobenius-Perron dimensions), Corollary \ref{chev} is proved in \cite{EG2}. In particular, this completely covers the case $p=3$, since in this case any tensor category with a fiber functor to ${\rm Ver}_p$ is integral. 

Moreover, Theorem \ref{subca} in the case of integral categories for any $p>2$ (i.e., when $\C$ has a fiber functor into $\sVec_\k$) is \cite{EG3}, Theorem 2.21(1). 
\end{remark} 

\section{Frobenius order}

\subsection{The case of general $p$.} 
Let $\C$ be a symmetric tensor category over $\k$ with finitely many simple objects. 
Recall that $\widetilde{\C}_{\rm ex}$ denotes the Serre closure of the category $\C_{\rm ex}$, a tensor subcategory of $\C$, which, according to Corollary \ref{serresubcat}, coincides with $\C_{\rm ex}$ for $p>2$. 

\begin{proposition}\label{index} There exists a natural number $n$ such that we have 
$$
F^n_\bullet(\C)\subset \widetilde{\C}_{\rm ex}.
$$
\end{proposition} 

\begin{proof} Let $I$ be the set of simple objects of $\C$ and $J\subset I$ be the set of simple objects of $\C_{\rm ex}$. 
For $k\ge 0$, let $I_k\subset I$ be the set of simple objects of $\C$ arising as composition factors 
of $F_\bullet^k(X)$, $X\in \C$. 

\begin{lemma}\label{notcon}
If $I_k$ is not contained in $J$ then $I_{k+1}$ is a proper subset of $I_k$. 
\end{lemma}

\begin{proof} Let $d={\rm max}_{Z\in I_k\setminus J}{\rm FPdim}(Z)$, and let $T\in J$ be such that 
${\rm FPdim}(T)=d$. Then we claim that $T\notin I_{k+1}$. Indeed, let $X\in \C$. 
Note that the set of composition factors of $F_\bullet(X)$ is the same as for $G_\bullet(X)$
(where the subscript runs over $[1,p]$). Since $G_i$ are exact in the middle, this implies that all composition factors of $F_\bullet^{k+1}(X)=F_\bullet(F_\bullet^k(X))$ are composition factors in $F_\bullet(Z)$, where $Z$ is a composition factor of $F_\bullet^k(X)$, i.e., $Z\in I_k$. 
But if $Z\in I_k\setminus J$ then ${\rm FPdim}(F(Z))<{\rm FPdim}(Z)$, so $F_i(Z)$ cannot contain $T$ as a 
composition factor for any $i$. If $Z\in J$ then $F(Z)\in \C_{\rm ex}$, so cannot contain $T$ as a composition factor either. 
Thus, $T\notin I_{k+1}$, as claimed.
\end{proof}

Lemma \ref{notcon} implies that there exists $n$ such that $I_n\subset J$. Then 
$$
F_\bullet^n(\C)\subset \widetilde{\C}_{\rm ex},
$$
 as desired. 
\end{proof} 

\begin{corollary} \label{index1} There exists a natural number 
$n$ such that we have 
$$
F^n_\bullet(\C)\subset \C_{\rm ex}.
$$
\end{corollary} 

\begin{proof} If $p>2$ this follows from Proposition \ref{index} and Corollary \ref{serresubcat}. 
If $p=2$ this follows from Proposition \ref{index} and Corollary \ref{cha2}.
\end{proof} 

\begin{definition} The smallest $n$ such that $F^n(\C)\subset \C_{\rm ex}$ is called 
{\it the Frobenius order} of $\C$ 
and denoted by $O(\C)$.
\end{definition} 

In particular, Frobenius exact categories are those with $O(\C)=0$. 

\begin{proposition}\label{pro1} If $E: \C\to \D$ is a symmetric tensor functor then 
$$
O(\C)\le O(\D).
$$ 
\end{proposition}

\begin{proof} This follows immediately from Proposition \ref{onlyif}(ii) and the fact that monoidal functors commute with the Frobenius functor (Proposition \ref{addmon}(ii)). 
\end{proof} 

\subsection{The case $p=2$} 
Consider now the case $p=2$. In this case we will call the Frobenius order $O(\C)$ 
the {\it upper Frobenius order}, and also make the following definition. 

\begin{definition} 
The {\it lower Frobenius order} of $\C$, denoted $o(\C)$, is the smallest $n$ such that the tensor subcategory of $\C$ generated by $F^n(\C)$ is almost Frobenius exact. 
\end{definition} 

\begin{proposition}\label{pro2} If $E: \C\to \D$ is a symmetric tensor functor then 
$$o(\C)\le o(\D).$$
\end{proposition} 

\begin{proof} Let $o(\D)=n$. Let $\C'$ be the tensor subcategory of $\C$ generated by $F^n(X)$ for $X\in \C$. 
For any $X\in \C$ we have $F^n(E(X))=E(F^n(X))\in \D'$, where $\D'$ is the tensor subcategory of $\D$ generated by $F^n(Y)$, $Y\in \D$. Thus, $E(Z)\in \D'$ for all $Z\in \C'$. But by definition, $\D'$ is almost Frobenius exact, so admits a fiber functor $L: \D'\to \V$. Then $L\circ E: \C'\to \V$ is a fiber functor. 
Thus $\C'$ is almost Frobenius exact, so $o(\C)\le n$, as claimed. 
\end{proof} 

Proposition \ref{alm} immediately implies 

\begin{corollary} We have $o(\C)\le O(\C)\le o(\C)+1$.  
\end{corollary}

Now let $\C_n,n\ge 0$ be the categories studied in \cite{BE}; 
e.g., $\C_0=\Vec_\k$ and $\C_1=\V$. 
By \cite{BE}, Corollary 2.6, the only tensor subcategories 
of $\C_n$ are $\C_m$ with $m\le n$. Hence 
$\C_{n\rm ex}=\C_0=\Vec_\k$ for $n\ge 0$, $\widetilde{\C}_{n\rm ex}=\C_1=\V$ for $n\ge 1$. 

\begin{proposition}\label{becat} We have $O(\C_n)=[\frac{n+1}{2}]$ and $o(\C_n)=[\frac{n}{2}]$. 
\end{proposition} 

\begin{proof} This follows from \cite{BE}, Remark 3.14, which implies that $F^n(\C_m)=\C_{m-2n}$ for any $n\le m/2$.  
\end{proof} 

Now let $\D$ be a symmetric tensor category with finitely many simple objects.

\begin{proposition}\label{critfib} (i) If $E: \D\to \C_{2n}$ is a fiber functor then we have $O(\D)\le n$;

(ii) If $E: \D\to \C_{2n+1}$ is a fiber functor then we have $o(\D)\le n$.
\end{proposition} 

\begin{proof} (i) This follows from Proposition \ref{pro1} and Proposition \ref{becat}. 

(ii) This follows from Proposition \ref{pro2} and Proposition \ref{becat}.
\end{proof} 

We expect that the converse to Proposition \ref{critfib} is true. Namely, we propose the following conjecture, which is a full analog of Deligne's theorem (\cite{De2}) for finite tensor categories in characteristic $2$. 

\begin{conjecture} \label{con2}
A finite tensor category $\D$ over $\k$ admits a fiber functor to $\C_{2n+1}$ if and only if $o(\D)\le n$ and 
to $\C_{2n}$ if and only if $O(\D)\le n$. 
\end{conjecture} 

\section{Appendix A. Hilbert series of graded commutative algebras in symmetric tensor categories.} 

The main goal of this appendix is to prove Corollary \ref{c3} which is used in the proof of Theorem \ref{Kth}. 

\begin{proposition}\label{p1} Let $\C$ be a symmetric tensor category with finitely many simple objects. Let $X\in \C$ and $d_i={\rm FPdim}(S^iX)$. 
Then the series $f_X(z)=\sum_{i\ge 0} d_iz^i$ has radius of convergence $1$ unless it is a polynomial. 
\end{proposition} 

\begin{proof} 
Since $d_i=0$ or $d_i\ge 1$ for all $i$, if $f_X(z)$ is not a polynomial then its radius of convergence 
$R\le 1$. So it suffices to show that $R\ge 1$.  

Since ${\rm FPdim}(S^i{\rm gr}X)\ge {\rm FPdim}({\rm gr}S^iX)= {\rm FPdim}(S^iX)$ for any filtered object $X$, it 
 suffices to prove the statement for $X$ being semisimple. 
 Thus it suffices to take $X$ to be the direct sum of all simple objects of $\C$. Then $S^2X$ 
has a filtration whose associated graded is contained in $mX$ for some integer $m$. Hence 
\begin{equation}\label{ineq1}
{\rm FPdim}S^{2i}X\le {\rm FPdim}S^iS^2X\le {\rm FPdim}(S^i(mX)).
\end{equation}
Let $f_X^\pm$ 
be the sum of all the even (resp. odd) terms of $f_X$. Then inequality \eqref{ineq1} implies 
$$
f_X^+(z)\le f_X(z^2)^m,
$$
where $\le $ is coefficientwise. Since $S^{2i+1}X$ is a quotient of $X\otimes S^{2i}X$, we also have 
$$
f_X^-(z)\le d_1zf_X(z^2)^m.
$$
Thus 
\begin{equation}\label{ineq2}
f_X(z)\le (1+d_1z)f_X(z^2)^m.
\end{equation} 
Let $R$ be the radius of convergence of $f_X$. Inequality \eqref{ineq2} implies that \linebreak $R\ge R^{1/2}$. Also $R\ge 1/d_1$ (as $S^iX$ is a quotient of $X^{\otimes i}$), so $R\ne 0$. 
 Thus $R\ge 1$, as claimed. 
\end{proof} 

\begin{corollary}\label{c2} Let $A=\oplus_{i\ge 0}A_i$ be a $\Bbb Z_+$-graded finitely generated commutative algebra in $\C$ with $A_0=\be$, and let $f(z)=\sum_{i\ge 0} {\rm FPdim}(A_i)z^i$. Then $f(z)$ has radius of convergence $1$ unless it is a polynomial. 
\end{corollary}  

\begin{proof} As before, it is clear that if $f$ is not a polynomial then its radius of convergence $R$ is $\le 1$. Thus the statement follows from Proposition \ref{p1} and the fact that $A$ is a quotient of $SX$ 
for some positively graded object $X\in \C$.
\end{proof} 


\begin{corollary}\label{c3} Let $A$ be a commutative ind-algebra in $\C$ such that 
$$
\Hom(\be,A)=K
$$ 
is
a field, and the pairing 
$$
\Hom_A(A,Y\otimes A)\times \Hom_A(Y\otimes A,A)\to \Hom_A(A,A)=K
$$
is nondegenerate for each $Y\in \C$. Then for each $Y\in \C$ we have 
$$
\dim_K\Hom_A(A,Y\otimes A)\le {\rm FPdim}(Y).
$$ 
\end{corollary}

\begin{proof} 
We will work in the category $\C_K$. 
Let $E\subset \Hom_A(A,Y\otimes A)$ be a finite dimensional $K-$subspace. Then we have a natural morphism 
$E\otimes A\to Y\otimes A$ which is split by nondegeneracy, in particular injective. Let $Z\subset A$ 
be a   
subobject such that $Z\in \C_K$ (as opposed to $Z\in \Ind(\C_K)$) and  
the image of $E\otimes \be$ in $Y \otimes A$ is contained in $Y\otimes Z$. Let $A_Z$ be the subalgebra of $A$ generated by $Z$. Then we have an injection $E\otimes A_Z\to Y\otimes A_Z$. Put a filtration on $A_Z$ by defining the degree of $Z$ to be $1$. 
Then we have an injection $E\otimes F^{i-1}A_Z\to Y\otimes F^iA_Z$ (as $Y$ sits in degree zero). Thus we have 
$$
\dim (E)zg(z)\le {\rm FPdim}(Y)g(z), 
$$
where $g(z):=\sum_{i\ge 0} {\rm FPdim}(F^iA_Z)z^i$. That is, we have  
\begin{equation}\label{equa1}
({\rm FPdim}(Y)-\dim (E)z)g(z)\ge 0. 
\end{equation}
Now, the algebra ${\rm gr}(A_Z)$ is finitely generated, 
and we have $g(z)=\frac{f(z)}{1-z}$, where $f(z):=\sum_{i\ge 0}{\rm FPdim}({\rm gr}(A_Z)_i)z^i$. 
 Thus by Corollary \ref{c2}, the left hand side of \eqref{equa1} converges in the open unit disk, hence can be evaluated at $z=1-\varepsilon$ for any $\varepsilon>0$. Thus we have 
$$
{\rm FPdim}(Y)-\dim (E)(1-\varepsilon)\ge 0. 
$$
Sending $\varepsilon$ to zero, we then get ${\rm FPdim}(Y)-\dim (E)\ge 0$, as claimed. 
\end{proof} 

\begin{conjecture}\label{con3} $f_X(z)$ is a rational function in $z$. 
\end{conjecture}

\begin{remark} Conjecture \ref{con3} clearly holds for $\C={\rm Ver}_p$, since the algebra $SL_i$ is finite dimensional for any $i>1$ (see e.g. \cite{EOV}). Thus it holds for any symmetric tensor category that fibers over ${\rm Ver}_p$, i.e. for Frobenius exact categories. But we don't know if it holds for the categories $\C_n$ 
from \cite{BE} for general $n$. 
\end{remark} 

\section{Appendix B: Simplicity and exactness of algebras in finite tensor categories}

The goal of this appendix is to prove Theorem \ref{NZ}, which is used in the proof of Theorem \ref{Kth}. 
It is also interesting in its own right, in particular provides a positive answer to Question 2.15 in \cite{EG}. 

Recall that an algebra $A$ in a finite tensor category $\C$ is called {\it exact} if the category ${}_A\C$ of left $A$-modules in $\C$ is an exact right $\C$-module category or, equivalently, the category $\C_A$ of right $A$-modules in $\C$ is an exact left $\C$-module category
 (\cite{EGNO}, Definition 7.8.20). 

\begin{theorem} \label{NZ}
Let $\C$ be a finite tensor category and let $A\in \C$ be an indecomposable
algebra. Then $A$ is exact if and only if the following conditions hold:

(a) $A$ is simple (i.e. it has no nontrivial two-sided ideals);

(b) there exists $X\in \C$ and an embedding of left $A-$modules 
$$
{}^*A\hookrightarrow A\ot X.
$$
\end{theorem}

\begin{proof} Assume that $A$ is exact. Then $A$ regarded as an $A-$bimodule is the unit
object of the dual category to $\C$ and hence is simple (see \cite[Lemma 3.24]{EO}); thus
(a) holds. Also injective objects in the category ${}_A\C$ are projective (see \cite[Corollary 3.6]{EO})
and any projective is a direct summand of $A\ot P$ for some projective object $P\in \C$; this implies (b).

Conversely, let $A$ be an algebra satisfying (a) and (b). We would like to show that for any 
$M\in \C_A$ and projective $P\in \C$ the object $P\ot M\in \C_A$ is projective. 

\begin{lemma}\label{1and2} (i) For any projective $P\in \C$ the object $P\ot A\in \C_A$ is injective.

(ii) If $N\in \C_A$ is nonzero then there is projective $P\in \C$ such that $P\ot N$ contains
a nonzero projective object of $\C_A$ as a direct summand.
\end{lemma} 

\begin{proof} (i) By (b) ${}^*A\ot {}^*P \subset A\ot X\ot {}^*P$. Since ${}^*A\ot {}^*P\in {}_A\C$ is injective,
${}^*A\ot {}^*P$ is a direct summand of $A\ot X\ot {}^*P$. Since $A\ot X\ot {}^*P\in {}_A\C$ is 
projective, ${}^*A\ot {}^*P$ is projective; dually, $P\ot A\in \C_A$ is injective. This proves (i). 

(ii) We have a map of $A-$bimodules $A\to {}^*N\ot N$. Using (a) we see that this map is injective.
Thus for any projective $0\ne Q\in \C$ we have an embedding of right $A-$modules
$Q\ot A\subset Q\ot {}^*N\ot N$. By (i), a projective object $Q\ot A\in \C_A$ is a direct summand
of $Q\ot {}^*N\ot N$, so we get the desired result with $P=Q\ot {}^*N$. This proves (ii). 
\end{proof} 

Let $\Gamma(\A)$ be the {\bf split} Grothendieck group of an additive
category $\A$; we also consider $\Gamma_\BR(\A):=\Gamma(\A)\ot_\BZ \BR$. Then $\Gamma(\C)$ is a ring acting on the group $\Gamma(\C_A)$. 

Let $\{ X_i\}_{i\in I}$ be the set of isomorphism classes of simple objects in $\C$. For any $i$ let $P_i\in \C$ be the projective cover of $X_i$. The following result is standard.

\begin{lemma}\label{iii} (\cite{EGNO}, Proposition 6.1.11) 
The element 
$$R:=\sum_{i\in I}\FPdim(X_i)[P_i]\in \Gamma_\BR(\C)$$
 satisfies 
$$[X]R=\FPdim(X)R\ \forall X\in \C.$$
\end{lemma} 


Now let $M\in \C_A$ and consider $R[M]\in \Gamma_\BR(\C_A)$. We can write 
$$
R[M]=\N+\cP
$$ 
  where $\N$ is a sum of the classes of indecomposable non-projective objects and 
$\cP$ is a sum of the classes of indecomposable projective objects.  
Assume that $\N \ne 0$. 
Let $P$ be as in Lemma \ref{1and2}(ii). Then $[P]\N$ will contain some projective summands; also
$[P]\cP$ is projective. Thus the contribution of projective summands into $\FPdim([P]R[M])$
is strictly more than 
$$
\FPdim([P]\cP)=\FPdim(P)\FPdim(\cP).
$$ 
But on the other hand by Lemma \ref{iii}
we have 
$$
[P]R[M]=\FPdim(P)R[M]=\FPdim(P)\N+\FPdim(P)\cP,
$$ 
so the contribution of
projective summands is precisely $\FPdim(P)\FPdim(\cP)$. This is a contradiction. Thus
$\N=0$, so $P_i\ot M$ is projective for any $i\in I$ and the theorem is proved.
\end{proof}

We now get a positive answer to Question 2.15 in \cite{EG}:

\begin{corollary} Let $\C$ be a tensor subcategory of a finite tensor category $\D$.
If $A\in \C$ is exact, then $A$ regarded as an algebra in $\D$ is also exact.
\end{corollary}

\begin{proof} Both conditions (a) and (b) of Theorem \ref{NZ} 
are clearly preserved by the inclusion functor $\C\hookrightarrow \D$.
\end{proof}

\begin{remark} We recall that an injective (i.e. fully faithful) tensor functor $\C \hookrightarrow \D$ is an equivalence
with a tensor (in particular, abelian) subcategory of $\D$, see \cite[Proposition 6.3.1]{EGNO}. Thus
a fully faithful tensor functor between finite tensor categories sends an exact algebra to an exact one.
\end{remark}

Observe that if condition (a) of Theorem \ref{NZ} holds, condition (b) is equivalent to the requirement that $A^*\ot_A{}^*A\ne 0$. If this fails then in the category $\C_A$ any morphism from an injective object to a projective object must vanish. Thus it seems plausible that condition (b) is a consequence of condition (a). In other words, one can make the following 

\begin{conjecture}\label{exactsimple} An indecomposable algebra $A\in \C$ is exact if and only if it is simple. 
\end{conjecture} 

Just the ``if" direction of Conjecture \ref{exactsimple} requires proof, as the ``only if" direction is provided by Theorem \ref{NZ}. 

\begin{proposition} Conjecture \ref{exactsimple} holds if $\C=\Rep H$ for a finite dimensional Hopf algebra $H$. 
\end{proposition} 

\begin{proof} Let $A\in \C$ be a simple algebra, and let us show that the right $\C$-module category $A-{\rm mod}_\C$ is exact. Since $\C=\Rep H$, we may view $A$ as an $H$-module algebra. We have $A-{\rm mod}_\C=A\rtimes H$-mod. 

   Now let $X\in A-{\rm mod}_\C=A\rtimes H$-mod. By \cite{Sk}, Theorem 3.5, the restriction functor 
$${\rm Res}: A\rtimes H-{\rm mod}\to A-{\rm mod}$$ lands in the subcategory of projective 
$A$-modules. Thus $X$ is a projective $A$-module. Hence $X\otimes H=X\otimes_A(A\rtimes H)$ is a projective $A\rtimes H$-module. Thus tensoring with $H$ takes any object $X\in A-{\rm mod}_\C$ to a projective object. 
Hence, $A-{\rm mod}_\C$ is an exact right $\C$-module category, i.e., $A$ is an exact algebra, as claimed. 

\end{proof}

\end{document}